\documentclass{amsart}
\usepackage{bbm}
\usepackage{cite}

 \usepackage{hyperref}
 \usepackage{amssymb}
 \usepackage{amsmath,amsthm}
 \usepackage{latexsym}
 \usepackage{amsfonts}
 \usepackage{graphicx}
 \usepackage{mathrsfs}
 \usepackage{comment}

\newtheorem{theorem}{Theorem}[section]

\newtheorem{corollary}[theorem]{Corollary}
\newtheorem{proposition}[theorem]{Proposition}

\theoremstyle{definition}
\newtheorem{definition}[theorem]{Definition}
\newtheorem{example}[theorem]{Example}

\theoremstyle{remark}
\newtheorem{remark}[theorem]{Remark}

\renewcommand{\d}{\delta}
\newcommand{\g}{\lambda}

\newcommand{\q}{\quad}

\newcommand{\s}{\sigma}

\newcommand{\M}{{\cal M}}

\newcommand{\ty}{\infty}
\newcommand{\e}{\varepsilon}

\newcommand{\ov}[1]{\overline{#1}}

\renewcommand{\O}{\Omega}

\newcommand{\eR}{\mathbb{R}}
\newcommand{\eN}{\mathbb{N}}
\newcommand{\Ze}{\mathbb{Z}}

\newcommand{\Ce}{\mathbb{C}}

\newcommand{\re}{\mathop{\mathrm{Re}}}

\newcommand{\res}{\operatorname{res}}

\newcommand{\I}{\mathbbm{i}}
\newcommand{\E}{\mathrm{e}}

\newcommand{\ovb}[1]{\mkern 1.5mu\overline{\mkern-1.5mu#1\mkern-1.5mu}\mkern 1.5mu}
\newcommand{\unb}[1]{\mkern 1.5mu\underline{\mkern-1.5mu#1\mkern-1.5mu}\mkern 1.5mu}

\newcommand{\di}{\,\mathrm{d}}

\newcommand{\eS}{\mathbb{S}}  
\newcommand{\arccot}{\operatorname{arccot}}

\newcommand{\cal}{\mathcal}

\begin{document}
\nocite{*}


\title{Fractal zeta functions at infinity and the $\phi$-shell Minkowski content}


\author[Goran Radunovi\'c]{Goran Radunovi\'c}
\thanks{The research of Goran Radunovi\'c was supported by Croatian Science Foundation (HRZZ) grant IP-2022-10-9820 and by the Horizon grant 101183111-DSYREKI-HORIZON-MSCA-2023-SE-01}
\address{University of Zagreb, Faculty of Science, Horvatovac 102a, 10000 Zagreb, Croatia}
\email{{\tt goran.radunovic@math.hr} (Goran Radunovi\'c)}

%
%
%

\begin{abstract}
We study fractal properties of unbounded domains with infinite Lebesgue measure via their complex fractal dimensions. These complex dimensions are defined as poles of a suitable defined Lapidus fractal zeta function at infinity and are a generalization of the Minkowski dimension for a special kind of a degenerated relative fractal drums at infinity. It is a natural generalization of a similar approach applied to unbounded domains of finite Lebesgue measures investigated previously by the author. In this case we adapt the definition of Minkowski content and dimension at infinity by introducing the so-called $\phi$-shells where $\phi$ is a real parameter.
We show that the new notion of the upper $\phi$-shell Minkowski dimension is independent on the parameter $\phi$ and well adapted to the fractal zeta function at infinity.
As an application we construct maximally hyperfractal and quasiperiodic domains at infinity of infinite Lebesgue measure.
We also reflect on how the new definition connects with the one-point compactification and the classical fractal properties of the corresponding compactified domain as well as to the notion of surface Minkowski content at infinity.
\end{abstract}



\maketitle

\tableofcontents


\section{Introduction}
The study of fractal properties of unbounded sets in Euclidean space presents unique challenges, particularly when considering their behavior at infinity. 
In order to study sets at infinity having infinite Lebesgue measure we introduce and analyze new geometrical notions: the $\phi$-shell Minkowski content and the corresponding $\phi$-shell Minkowski dimension at infinity.
We will show how these new notions generalize the results of \cite{ra2,ra3} where unbounded sets were required to have finite Lebesgue measure and, especially, from the point of view of the theory of complex dimensions and the related fractal zeta functions at infinity.

As an application of the new notions we demonstrate how the construction of maximally hyperfractal unbounded domains and quasiperiodic unbounded domains can be done even without the assumption of them having finite Lebesgue measure, as opposed to \cite{ra3}
%

In concluding remarks we reflect how the notion of $\phi$-Minkowski content at infinity relates to the classical notion of Minkowski content under the one-point compactification via the stereographic projection and also to the natural notion of surface Minkowski content at infinity.

In this paper we are focused on a special kind of degenerated pairs of subsets (called relative fractal drums, or RFDs, in short \cite{fzf}) of the $N$-dimensional Euclidean space $\mathbb{R}^N$, $(\infty,\Omega)$ where $\infty$ is the point at infinity. As opposed to our previous work on the topic, here we remove the assumption on the set $\Omega$ of having finite Lebesgue measure. Hence we need to modify the definition of the Minkowski dimension at infinity by introducing a parameter $\phi>1$ and focusing on shell-like parts of the parallel body ($t$-neighborhood of infinity intersecting $\Omega$) of ``width'' $\phi$.

 Recall that complex dimensions are a far-reaching generalization of the classical Minkowski dimension for compact subsets of $\eR^N$, defined analytically as poles of suitably defined (Lapidus) fractal zeta functions.
 In the classical setting, relative fractal drums (RFDs) $(A,\Omega)$ where $A\subseteq\eR^N$ is a nonempty subset and $\Omega\subseteq\eR^N$ has finite $N$-dimensional volume generalize the notion of compact subsets of $\mathbb{R}^N$. These have been extensively studied in \cite{LapRaZu4,LapRaZu6,LapRaZu7,fzf}, along with their associated Minkowski dimension, content, and complex dimensions defined via (Lapidus) fractal zeta functions. The distance zeta function of an RFD $(A,\Omega)$ is defined as: \begin{equation}\label{zeta_dist}
\zeta_{A,\Omega}(s):=\int_{\Omega}\mathrm{dist}(x,A)^{s-N}\,\mathrm{d}x,
\end{equation} initially for $s\in\mathbb{C}$ with $\mathrm{Re}(s)$ sufficiently large for guaranteeing absolute convergence, while $\mathrm{dist}(x,A)$ is the Euclidean distance from $x$ to $A$. The higher-dimensional theory of complex dimensions \cite{fzf,LapRaZu2,LapRaZu3,brezish,tabarz} also generalizes the theory of geometric zeta functions for fractal strings developed by Lapidus, van Frankenhuijsen \cite{lapidusfrank12}.
In order to study unbounded sets from this perspective, basic notions of Minkowski content, dimension, and zeta functions at infinity were introduced in \cite{ra,ra2}. 

The ``fractality'' of $(\infty,\Omega)$ stems solely from the unbounded set $\Omega$, in contrast to classical RFDs where it can arises from both $A$ and $\Omega$ but in which case the set $A$ is usually the one being of interest, while the set $\O$ serves for localization purposes and is usually kept metrically associated\footnote{Meaning that $\O$ is a preimage of a subset of $A$ under the metric projection.} to the set $A$.
Note that it was shown in \cite{ra,ra2} that degenerated RFDs $(\infty,\Omega)$ having nontrivial Minkowski dimension exist, hence, complicated sets $\O$ can have a serious impact on the ``fractality'' of the RFD $(A,\O)$ in general.
However, we can use this effect in order to study unbounded domains from a fractal geometric point of view.

Motivation for studying unbounded fractal domains comes from oscillation theory \cite{Dzu,Karp}, automotive engineering \cite{She}, aerodynamics \cite{Dol}, civil engineering \cite{Pou}, mathematical biology \cite{May}, and PDEs in unbounded domains \cite{Maz1,mazja,An,Hur,La,Rab,VoGoLat}. Furthermore, fractal properties of unbounded trajectories in planar vector fields were studied in \cite{razuzup} related to Hopf bifurcation at infinity. 

The paper is structured as follows.
Section \ref{phi_sec} introduces the $\phi$-shell Minkowski content and the corresponding $\phi$-shell Minkowski dimension at infinity.
In case of unbounded domains that have finite Lebesgue measure, we show that these new parametric notions generalize the notion of Minkowski dimension at infinity (introduced in \cite{ra2,ra3}) by obtaining comparison results and studying limiting behavior of the parameter $\phi$. 
The possible range for the $\phi$-shell Minkowski dimension is shown to be $[-\infty, 0]$ while the range for unbounded domains of infinite Lebesgue measure is shown to be $[-N,0]$ filling the dimensional gap observed in \cite{ra2} for unbounded sets of finite Lebesgue measure.
We also study dependence on the parameter $\phi$ and show that the upper $\phi$-shell Minkowski dimension is independent on the choice of $\phi$ while this is not the case for the lower  $\phi$-shell Minkowski dimension.

In Section \ref{phi_zeta} we show that the results about the Lapidus zeta functions of subsets of finite Lebesgue measure at infinity generalize to the case of subsets with infinite Lebesgue measure.
The first main theorem of the paper is the holomorphicity Theorem  \ref{analiticinf_inf} which generalizes \cite[Thm.\ 5]{ra2} \cite[Thm.\ 5]{ra2} by using the $\phi$-shell Minkowski dimension.
The second main Theorem \ref{resdistinf_inf} then shows that under natural assumptions on the unbounded set $\O$, the residue of its Lapidus zeta function at infinity evaluated at its upper $\phi$-shell Minkowski dimension is bounded from above (below) by its upper (lower) $\phi$-shell Minkowski content normalized by an appropriate constant which is a generalization of \cite[Theorem 6]{ra2}.

In Section \ref{sec:quasi} we describe a  construction of interesting nontrivial examples of unbounded domains of infinite Lebesgue measure having maximally hyperfractal and quasiperiodic behavior at infinity which generalize examples from \cite{ra3}.

In Section \ref{surface_sec}, we briefly reflect on how the notion of $\phi$-shell Minkowski content connects to one-point compactification via the spherical Minkowski content and also to the natural notion of surface Minkowski content at infinity.
We obtain some comparison results and comment on open problems.

\section{The $\phi$-shell Min\-kowski content at infinity}\label{phi_sec}

We start by introducing the notion of the parametric $\phi$-shell Minkowski content at infinity which asks for the asymptotic behavior of the volume of the intersection of shells (which have thickens prescribed by the parameter $\phi>0$) and the unbounded set $\O$ as the inner radius of the shell grows to infinity.
More, precisely, for $t>0$, we denote by $B_t(0):=\{x\in\eR^N:|x|<t\}$ the open ball centered at origin of radius $t$.
Next, for any $a,b\in\eR$, with $b>a>0$, we define the $[a,b)$-{\it shell} centered at the origin as $B_{a,b}(0):=B_a(0)^c\cap B_b(0) =\{x\in\eR^N:a\leq|x|<b\}$.

\begin{definition}[$\phi$-shell Minkowski content at infinity]\label{phi_mink_def}
Let $\O\subseteq\eR^N$ be a Lebesgue measurable set,\footnote{Note that here we do not require that $\O$ has finite Lebesgue measure.} $\phi>1$ and $r\in\eR$.
We define the {\em upper $\phi$-shell Minkowski content} of $\O$ {\em at infinity} as
\begin{equation}\label{G_phi_shell}
{\ovb{\M}}_{\phi}^{r}(\ty,\O):=\limsup_{t\to +\ty}\frac{|B_{t,\phi t}(0)\cap\O|}{t^{N+r}}
\end{equation}
and analogously the {\em lower $\phi$-shell Minkowski content} of $\O$ {\em at infinity} as
\begin{equation}\label{D_phi_shell}
{\unb{\M}}_{\phi}^{r}(\ty,\O):=\liminf_{t\to +\ty}\frac{|B_{t,\phi t}(0)\cap\O|}{t^{N+r}}.
\end{equation}
If for some $r\in\eR$ the upper and lower limits above coincide, we call this value the {\em$r$-dimensional $\phi$-shell Minkowski content of $\O$ at infinity} and denote it with ${\mathcal{M}}_\phi^{r}(\ty,\O)$.
Furthermore, we will call the function $t\mapsto|B_{t,\phi t}(0)\cap\O|$ the {\em $\phi$-shell function} of $\O$.
\end{definition}

Recall from \cite{ra,ra2} that for unbounded sets of finite Lebesgue measure we already have a well defined notion of Minkowski content at infinity.
Namely, the {\it upper $r$-dimensional Minkowski content at infinity} is defined as 
\begin{equation}
	{\ovb{\M}}^{r}(\ty,\O):=\limsup_{t\to+\infty} \frac{|B_t(0)^c\cap\O|}{t^{N+r}}
\end{equation}
and similarly for the lower counterpart. 
As usual, then the {\it upper (lower) Minkowski dimension of $\O$ at infinity}, denoted by $\ovb{\dim}_B(\infty,\O)$ ($\unb{\dim}_B(\infty,\O)$), is defined in the standard way - as the critical value of the exponent $r$ for which the value of the upper (lower) Minkowski content at infinity drops from $+\infty$ to 0.
Recall also that $-\infty\leq\unb{\dim}_B(\infty,\O)\leq\ovb{\dim}_B(\infty,\O)\leq -N$ for any unbounded set $\O$ of finite Lebesgue measure; \cite{ra2}.

In the case when $\O\subseteq\eR^N$ is unbounded and of finite Lebesgue measure, the next proposition gives a comparison result between the just introduced $\phi$-shell Minkowski content and the Minkowski content at infinity. 

\begin{proposition}\label{prstenasta}
Let $\O\subseteq\eR^N$ be a Lebesgue measurable set with $|\O|<\ty$. Then, for every $\phi>1$ and $r<-N$ we have
\begin{equation}\label{G_prsten}
{\ovb{\M}}_{\phi}^{r}(\ty,\O)\leq{\ovb{\M}}^{r}(\ty,\O)\leq\frac{1}{1-\phi^{N+r}}\,{\ovb{\M}}_{\phi}^{r}(\ty,\O)
\end{equation}
and
\begin{equation}\label{D_prsten}
\frac{1}{1-\phi^{N+r}}\,{\unb{\M}}_{\phi}^{r}(\ty,\O)\leq{\unb{\M}}^{r}(\ty,\O).
\end{equation}
\end{proposition}

\begin{proof}
The left-hand side of inequality~\eqref{G_prsten} is a simple consequence of the fact that $|B_{t,\phi t}(0)\cap\O|\leq|B_t(0)^c\cap \O|$.
To prove the rest of the proposition, we observe that
$$
|B_{t}(0)^c\cap\O|=\sum_{n=0}^{\ty}|B_{\phi^n t,\phi^{n+1} t}(0)\cap\O|,
$$
which, in turn, implies that for $r<-N$ we have
$$
\frac{|B_{t}(0)^c\cap\O|}{t^{N+r}}=\sum_{n=0}^{\ty}\phi^{n(N+r)}\frac{|B_{\phi^n t,\phi^{n+1} t}(0)\cap\O|}{(\phi^n t)^{N+r}}.
$$
We now apply Fatou's lemma\footnote{More precisely, first we choose an arbitrary sequence of positive numbers $(t_k)_{k\geq 1}$ such that $t_k\to +\ty$ and apply Fatou's lemma on the counting measure in this case. From that we get the conclusion in the general case when $t\to +\ty$.} and the fact that $\phi^{N+r}<1$ to get
$$
\begin{aligned}
\limsup_{t\to+\ty}\frac{|B_{t}(0)^c\cap\O|}{t^{N+r}}&\leq\sum_{n=0}^{\ty}\phi^{n(N+r)}\limsup_{t\to+\ty}\frac{|B_{\phi^n t,\phi^{n+1} t}(0)\cap\O|}{(\phi^n t)^{N+r}}\\
{\ovb{\M}}^{r}(\ty,\O)&\leq\sum_{n=0}^{\ty}\phi^{n(N+r)}{\ovb{\M}}^{r}_{\phi}(\ty,\O)=\frac{1}{1-\phi^{N+r}}\,{\ovb{\M}}^{r}_{\phi}(\ty,\O).
\end{aligned}
$$
Finally, by the same reasoning applied to the lower limit we get~\eqref{D_prsten} and this concludes the proof of the proposition.
\end{proof}

Since $\phi\mapsto{\unb\M}_{\phi}^{r}(\ty,\O)$ and $\phi\mapsto{\ovb{\M}}_{\phi}^{r}(\ty,\O)$ are nondecreasing functions with values in $[0,+\ty]$, the next corollary follows immediately from the above proposition.

\begin{corollary}
Let $\O\subseteq\eR^N$ be a Lebesgue measurable set with $|\O|<\ty$. Then, for $r<-N$ we have that
\begin{equation}
\lim_{\phi\to+\ty}{\ovb{\M}}_{\phi}^{r}(\ty,\O)={\ovb{\M}}^{r}(\ty,\O)
\end{equation}
and
\begin{equation}
\lim_{\phi\to +\ty}{\unb{\M}}_{\phi}^{r}(\ty,\O)\leq{\unb{\M}}^{r}(\ty,\O)
\end{equation}
\end{corollary}

\begin{remark}\label{the_remark}
In light of the above corollary, a valid question is to interpret the meaning of $\lim_{\phi\to 1^+}{\unb\M}_{\phi}^{r}(\ty,\O)$ and $\lim_{\phi\to 1^+}{\ovb{\M}}_{\phi}^{r}(\ty,\O)$.
Moreover, if ${\unb{\M}}^{r}(\ty,\O)<\ty$ we have from Proposition~\ref{prstenasta} that $\lim_{\phi\to 1^+}{\unb{\M}}_{\phi}^{r}(\ty,\O)=0.$
One would expect that these limits are somehow related to the notion of the {\em surface Minkowski content} investigated by Winter and Rataj in~\cite{RaWi1} and~\cite{winter}.
We will reflect on this in Section~\ref{surface_sec} below.
\end{remark}

The introduction of the $\phi$-shell Minkowski content at infinity raises now the definition of an appropriate notion of a $\phi$-shell Minkowski dimension at infinity in the classical way. 

\begin{definition}[$\phi$-shell Minkowski dimension at infinity]\label{phi_shell_box_def}
Let $\O\subseteq\eR^N$ be a Lebesgue measurable set and let $\phi>1$. 
Now, we can define the {\em upper} {\em $\phi$-shell Minkowski dimension} of $\O$ {\em at infinity}$:$
\begin{equation}\label{upperphiD}
\begin{aligned}
{\ovb{\dim}}_B^{\phi}(\ty,\O)&:=\sup\{r\in\eR\,:\,{\ovb{\mathcal{M}}}_{\phi}^{r}(\ty,\O)=+\ty\}\\
&\phantom{:}=\inf\{r\in\eR\,:\,{\ovb{\mathcal{M}}}_{\phi}^{r}(\ty,\O)=0\};
\end{aligned}
\end{equation}
and analogously the  {\em lower} counterpart denoted by ${\unb{\dim}}_{B}^{\phi}(\ty,\O)$.
\end{definition}

The next example will show that for every Lebesgue measurable set $\O\subseteq\eR^N$ and for every $\phi>1$ we have that ${\unb{\dim}}_{B}^{\phi}(\ty,\O)\le{\ovb{\dim}}_{B}^{\phi}(\ty,\O)\leq 0$.
Note that the fact that ${\ovb{\dim}}_{B}^{\phi}(\ty,\O)\leq 0$ is in accord with the fact that the one-point set $\{\mathbf{N}\}$ has spherical upper $\phi$-shell Minkowski dimension relative to any subset of $\eS^N$ maximally equal to $0$.
We will reflect on this in Subsection \ref{subsec:cpt} in more detail.

\begin{example}\label{er_N}
Let $\O=\eR^N$ and $\phi>1$.
Then we have that $\dim_B^{\phi}(\ty,\eR^N)=0$.
This follows from
$$
|B_{t,\phi t}(0)\cap\eR^N|=|B_{\phi t}(0)|-|B_t(0)|=\frac{\pi^{\frac{N}{2}}(\phi^N-1)t^N}{\Gamma\left(\frac{N}{2}+1\right)}.
$$
Moreover, we have that
$
\M_{\phi}^0(\ty,\eR^N)=\frac{\pi^{\frac{N}{2}}(\phi^N-1)}{\Gamma\left(\frac{N}{2}+1\right)}.
$
\end{example}

As a consequence of the above example we immediately get the following proposition which gives bounds on possible values of the upper and lower $\phi$-shell Minkowski dimension at infinity.

\begin{proposition}\label{<=0}
Let $\O$ be a Lebesgue measurable subset of $\eR^N$ and $\phi>1$.
Then the upper and lower $\phi$-shell Minkowski dimensions of $\O$ at infinity are always nonpositive, i.e.,
$$
{\unb{{\dim}}}_{B}^{\phi}(\ty,\O)\leq{{{\ovb{\dim}}}}_{B}^{\phi}(\ty,\O)\leq 0.
$$
\end{proposition}

If we additionally assume that the unbounded set $\O$ has infinite Lebesgue measure, then this is reflected in its $\phi$-shell Minkowski at infinity as follows.

\begin{proposition}\label{>=-N}
Let $\O$ be a Lebesgue measurable subset of $\eR^N$ with $|\O|=\ty$.
Then for every $\phi>1$ we have$:$
$$
-N\leq{{\ovb{\dim}}}_B^{\phi}(\ty,\O)\leq 0.
$$
\end{proposition}

\begin{proof}
We reason by contradiction, i.e., we assume that there exists $\phi>0$ such that ${{\ovb{\dim}}}_B^{\phi}(\ty,\O)<-N$.
Then we fix $\sigma\in({{\ovb{\dim}}}_B^{\phi}(\ty,\O),-N)$ and take $T$ large enough such that there exists a constant $M>0$ and
$$
|B_{t,\phi t}(0)\cap\O|\leq Mt^{\sigma+N}
$$
for every $t>T$.
Now we have
$$
|B_T(0)^c\cap\O|=\sum_{n=0}^{\ty}|B_{\phi^{n}T,\phi^{n+1}T}|\leq M\sum_{n=0}^{\ty}(\phi^nT)^{\sigma+N}=\frac{MT^{\sigma+N}}{1-\phi^{\s+N}}<\ty,
$$
since $\s+N<0$ which contradicts the fact that $|\O|=\ty$.
\end{proof}

The statement of the above proposition is optimal, i.e., there are sets of infinite volume with upper $\phi$-shell Minkowski dimension equal to $-N$.
This illustrates the next example in $\eR^2$ and can be easily adapted in the case of $\eR^N$.

\begin{example}
Let $\O:=\{(x,y)\in\eR^2\,:\,x>1,\ 0<y<x^{-1}\}$.
Then, for any $\phi>1$ and $t>1$ it is not difficult to see that have
$$
|B_{t,\phi t}(0)\cap\O|\asymp\int_{t}^{\phi t}\frac{1}{x}\di x=\log(\phi t)-\log t=\log\phi,
$$
where $\asymp$ means that the ratio of left and right side tends to 1, as $t\to +\infty$.
From this we see that ${{\dim}}_B^{\phi}(\ty,\O)=-2$ and $\M_{\phi}^{-2}(\ty,\O)=\log\phi$.
%
%
%
\end{example}

Note that from Proposition~\ref{prstenasta} we immediately obtain the next comparison result concerning sets $\O$ of finite Lebesgue measure and their $\phi$-shell Minkowski dimensions.

\begin{corollary}\label{phidim}
Let $\O\subseteq\eR^N$ be of finite Lebesgue measure such that ${\ovb{\dim}}_B(\ty,\O)<-N$.
Then for every $\phi>1$ we have that
$$
{\ovb{\dim}}_B^{\phi}(\ty,\O)={\ovb{\dim}}_B(\ty,\O)
$$
and
$$
{\unb{\dim}}_{B}^{\phi}(\ty,\O)\leq{\unb{\dim}}_{B}(\ty,\O).
$$
Furthermore, if $D:={\dim}_{B}^{\phi}(\ty,\O)$ exists, then $\dim_{B}(\ty,\O)$ exists and in that case we have
$$
D=\dim_{B}(\ty,\O)=\dim_{B}^{\phi}(\ty,\O).
$$
Moreover, if $\O$ is $\phi$-shell Minkowski measurable at infinity, then it is Minkowski measurable at infinity and in that case we have
$$
{\mathcal{M}}^D(\ty,\O)=\frac{1}{1-\phi^{N+r}}\,{\mathcal{M}}_{\phi}^D(\ty,\O).
$$
\end{corollary}

The analog of Corollary~\ref{phidim} for the general case when we do not require $\O$ to be of finite Lebesgue measure still holds.
This is the statement of the next proposition that will show that the new notion of the (upper) $\phi$-shell Minkowski dimension at infinity is essentially independent of the choice of $\phi>1$.
This is not true for its lower counterpart as we will see in the example provided after the proposition.

\begin{proposition}\label{phi_independent}
Let $\O\subseteq\eR^N$ be a Lebesgue measurable set and $\phi_1,\phi_2\in\eR$ such that $1<\phi_1<\phi_2$.
Then, for any $r\in\eR\setminus\{-N\}$ we have$:$
\begin{equation}\label{prva_lijevo}
{\ovb{\M}}_{\phi_1}^{r}(\ty,\O)\leq{\ovb{\M}}_{\phi_2}^{r}(\ty,\O)\leq\frac{1-\phi_1^{(N+r)(\lfloor\log_{\phi_1}\phi_2\rfloor+1)}}{1-\phi_1^{N+r}}{\ovb{\M}}_{\phi_1}^{r}(\ty,\O)
\end{equation}
and
\begin{equation}\label{druga_lijevo}
\frac{1-\phi_1^{(N+r)\lfloor\log_{\phi_1}\phi_2\rfloor}}{1-\phi_1^{N+r}}{\unb{\M}}_{\phi_1}^{r}(\ty,\O)\leq{\unb{\M}}_{\phi_2}^{r}(\ty,\O),
\end{equation}
where $\lfloor\cdot\rfloor$ is the floor function.
Furthermore, in the case when $r=-N$ we have$:$
\begin{equation}
{\ovb{\M}}_{\phi_1}^{-N}(\ty,\O)\leq{\ovb{\M}}_{\phi_2}^{-N}(\ty,\O)\leq(\lfloor\log_{\phi_1}\phi_2\rfloor+1){\ovb{\M}}_{\phi_1}^{-N}(\ty,\O)
\end{equation}
and
\begin{equation}
\lfloor\log_{\phi_1}\phi_2\rfloor{\unb{\M}}_{\phi_1}^{-N}(\ty,\O)\leq{\unb{\M}}_{\phi_2}^{-N}(\ty,\O).
\end{equation}
Moreover, for the $\phi$-shell Minkowski dimensions at infinity we have$:$
\begin{equation}
{\ovb{\dim}}_{B}^{\phi_1}(\ty,\O)={\ovb{\dim}}_{B}^{\phi_2}(\ty,\O)
\end{equation}
and
\begin{equation}
{\unb{\dim}}_{B}^{\phi_1}(\ty,\O)\leq{\unb{\dim}}_{B}^{\phi_2}(\ty,\O).
\end{equation}

Finally, if $D:={\dim}_{B}^{\phi_1}(\ty,\O)$ exists, then ${\dim}_{B}^{\phi_2}(\ty,\O)$ exists as well and in that case we have
$$
D={\dim}_{B}^{\phi_1}(\ty,\O)={\dim}_{B}^{\phi_2}(\ty,\O).
$$
In addition, if $\O$ is $\phi_1$-shell Minkowski measurable at infinity, then it is $\phi_2$-shell Minkowski measurable at infinity.
\end{proposition}

\begin{proof}
Firstly, we observe that the left-hand part of~\eqref{prva_lijevo} is a simple consequence of the fact that $|B_{t,\phi_1t}(0)\cap\O|\leq|B_{t,\phi_2t}(0)\cap\O|$.
Secondly, it is easy to see that
$$
\frac{1-\phi_1^{(N+r)\lfloor\log_{\phi_1}\phi_2\rfloor}}{1-\phi_1^{N+r}}>1
$$
is fulfilled regardless of the sign of $N+r\neq 0$.
Consequently, this factor gives us a better estimate in~\eqref{druga_lijevo} than using the same argument as for~\eqref{prva_lijevo}.
Now, we let $k:=\lfloor\log_{\phi_1}\phi_2\rfloor$ and observe that
$$
\sum_{n=0}^{k-1}|B_{\phi_1^nt,\phi_1^{n+1}t}(0)\cap\O|\leq|B_{t,\phi_2t}(0)\cap\O|\leq\sum_{n=0}^{k}|B_{\phi_1^nt,\phi_1^{n+1}t}(0)\cap\O|.
$$
Furthermore, from this we get that
$$
\sum_{n=0}^{k-1}\phi_1^{n(N+r)}\frac{|B_{\phi_1^nt,\phi_1^{n+1}t}(0)\cap\O|}{(\phi_1^nt)^{N+r}}\leq\frac{|B_{t,\phi_2t}(0)\cap\O|}{t^{N+r}}\leq\sum_{n=0}^{k}\phi_1^{n(N+r)}\frac{|B_{\phi_1^nt,\phi_1^{n+1}t}(0)\cap\O|}{(\phi_1^nt)^{N+r}}.
$$
Finally, taking the upper and lower limits when $t\to+\ty$ gives us
$$
{\ovb{\M}}_{\phi_2}^{r}(\ty,\O)\leq{\ovb{\M}}_{\phi_1}^{r}(\ty,\O)\sum_{n=0}^{k}\phi_1^{n(N+r)}=
\begin{cases}
\frac{1-\phi_1^{(N+r)(\lfloor\log_{\phi_1}\phi_2\rfloor+1)}}{1-\phi_1^{N+r}}{\ovb{\M}}_{\phi_1}^{r}(\ty,\O),&r\neq -N\\
(\lfloor\log_{\phi_1}\phi_2\rfloor+1){\ovb{\M}}_{\phi_1}^{-N}(\ty,\O),&r=-N
\end{cases}
$$
and
$$
{\unb{\M}}_{\phi_2}^{r}(\ty,\O)\geq{\unb{\M}}_{\phi_1}^{r}(\ty,\O)\sum_{n=0}^{k-1}\phi_1^{n(N+r)}=
\begin{cases}
\frac{1-\phi_1^{(N+r)\lfloor\log_{\phi_1}\phi_2\rfloor}}{1-\phi_1^{N+r}}{\unb{\M}}_{\phi_1}^{r}(\ty,\O),&r\neq -N\\
\lfloor\log_{\phi_1}\phi_2\rfloor{\unb{\M}}_{\phi_1}^{-N}(\ty,\O),&r=-N.
\end{cases}
$$
\end{proof}

The next example demonstrates that the lower $\phi$-shell Minkowski dimension can depend on the parameter $\phi$ in general. 

\begin{example}\label{donji_kontraprim}
Let us fix a number $q>0$ and define $\O\subseteq\eR$ as a disjoint union of intervals:
$$
\O:=\bigcup_{n=0}^{\ty}\left(2^{2n+1},2^{2n+1}+\frac{1}{2^{2nq}}\right).
$$
Note that $|\O|=\sum_{n=0}^{\ty}2^{-2nq}=4^q/(4^q-1)$.
We take $\phi=2$ and  observe that for the sequence $t_n:=2^{2n}$, where $n\geq 0$ we have that $|B_{t_n,2t_n}(0)\cap\O|=0$.
This implies that ${\unb{\M}}_2^{r}(\ty,\O)=0$ for every $r\in\eR$, and, consequently, $\unb{\dim}_B^{2}(\ty,\O)=-\ty$.
On the other hand, if we take $\phi=4$, we have for $n\in\eN\cup\{0\}$ that
$$
|B_{t,4t}(0)\cap\O|=
\begin{cases}
4^{-nq},&t\in[2^{2n},2^{2n+1}]\\
4^{-nq}+3(t-2^{2n+1}),&t\in[2^{2n+1},2^{2n+1}+4^{-(n+1)q-1}]\\
4^{-nq}+4^{-(n+1)q}+2^{2n+1}-t,&t\in[2^{2n+1}+4^{-(n+1)q-1},2^{2n+1}+4^{-nq}]\\
4^{-(n+1)q},&t\in[2^{2n+1}+4^{-nq},2^{2(n+1)}].
\end{cases}
$$
As we can see, the $4$-shell function is constant on the intervals of the first and fourth type above, and linear on the intervals of the second and third type.

\begin{figure}[h]
\begin{center}
\includegraphics[width=11cm]{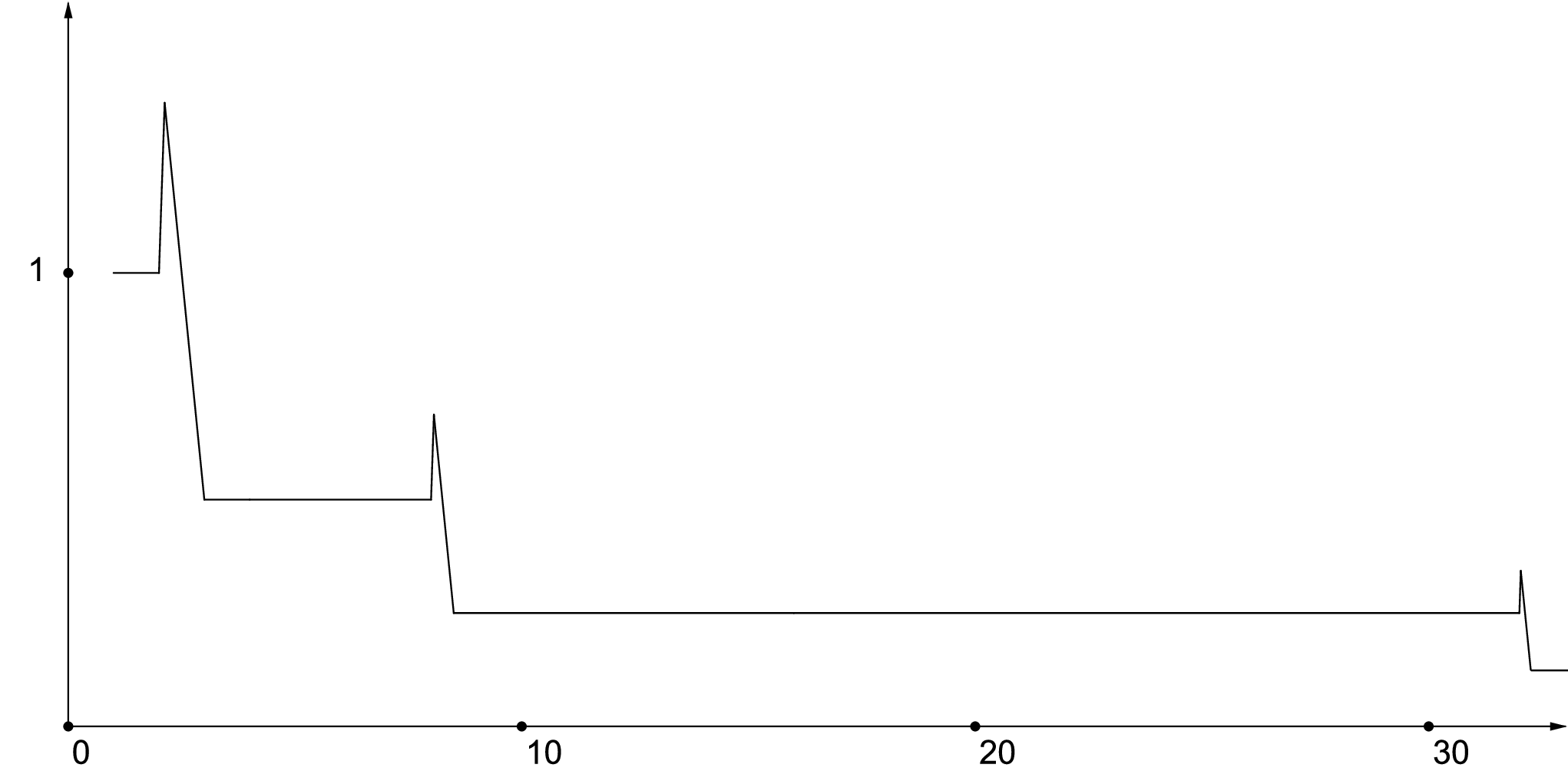}
\end{center}
\caption{A plot of the function $t\mapsto|B_{t,4t}(0)\cap\O|$ from Example \ref{donji_kontraprim}. Here, the parameter $q$ is equal to $1/2$. Note that the axes are not in the same scale.}
\label{tent}
\end{figure} 
In other words, it is a kind of a step function with `tents' between every two steps (see Figure \ref{tent}): 
$$
|B_{t,4t}(0)\cap\O|=
\begin{cases}
2^{q\{\log_2t\}}t^{-q},&t\in[2^{2n},2^{2n+1}]\\
2^{q(1-\{\log_2t\})}t^{-q}+3(t-2^{2n+1}),&t\in[2^{2n+1},2^{2n+1}+4^{-(n+1)q-1}]\\
2^{-q\{\log_2t\}}(2^q+2^{-q})t^{-q}+2^{2n+1}-t,&t\in[2^{2n+1}+4^{-(n+1)q-1},2^{2n+1}+4^{-nq}]\\
2^{q(-1-\{\log_2t\})}t^{-q},&t\in[2^{2n+1}+4^{-(n+1)q-1},2^{2(n+1)}],
\end{cases}
$$
where $\{\cdot\}$ denotes the fractional part function.
From this we have
$$
\frac{1}{t^{q}}\leq|B_{t,4t}(0)\cap\O|\leq\frac{2^q(1+3\cdot 4^{-q-1})}{t^{q}}
$$
which, in turn, implies that $\dim_B^4(\ty,\O)=-1-q$ and
$$
1\leq{\unb{\M}}_4^{-1-q}(\ty,\O)\leq{\ovb{\M}}_4^{-1-q}(\ty,\O)\leq 2^q(1+3\cdot 4^{-q-1}).
$$
This demonstrates that the conclusions of Proposition~\ref{phi_independent} concerning the lower $\phi$-shell Minkowski content and $\phi$-shell Minkowski dimension at infinity cannot be improved in general\footnote{A similar example can be constructed in $\eR^N$ by using the shells of appropriate radii of the $N$-dimensional ball centered at the origin.\label{remm}} since we have
$$
0={\unb{\M}}_2^{-1-q}(\ty,\O)<{\unb{\M}}_4^{-1-q}(\ty,\O)
$$
and
$$
-\ty=\unb{\dim}_B^{2}(\ty,\O)<{\ovb{\dim}}_B^2(\ty,\O)=\dim_B^4(\ty,\O)=-1-q.
$$
\end{example}

\section{Lapidus zeta functions at infinity and the $\phi$-shell Minkowski content}\label{phi_zeta}

In this section we will show that the results about the Lapidus zeta functions of subsets of finite Lebesgue measure at infinity studied in \cite{ra2,ra3} can be generalized to the case of subsets that do not have finite Lebesgue measure.
The generalization will be made by using the notions of the $\phi$-shell Minkowski contents and Minkowski dimensions at infinity.
To this end we will need the following result.

\begin{proposition}\label{integralna_veza_inf}
Let $\O\subseteq\eR^N$ be a Lebesgue measurable set, $T>0$ and $u\colon(T,+\ty)\to[0,+\ty)$ a strictly decreasing $C^1$ function such that $u(t)\to0$ as $t\to +\ty$. Then, the following equality holds
\begin{equation}\label{int_vez_inf}
\int_{{B_T(0)^c\cap\O}}u(|x|)\di x=\int_T^{+\ty}|{B_{T,t}(0)\cap\O}||u'(t)|\di t.
\end{equation}
\end{proposition}

\begin{proof}
We will use the well known fact which we recall here:
\begin{equation}\label{slojevita_2}
\int_Xf(x)\di x=\int_0^{\infty}|\{x\in X\,:\, f(x)\geq t\}|\di t,
\end{equation}
where $f$ is a nonnegative Borel function on a separable metric space $X$ (see, e.g., \cite{rudin}).
We let $f(x):=u(|x|)$  and $X:={{B_T(0)^c\cap\O}}$.
By assumption $u$ is strictly decreasing and $u(+\infty):=\lim_{\tau\to+\ty}u(\tau)=0$.
For the set appearing on the right side of \eqref{slojevita_2} we have
$$
A(t):=\{x\in {B_T(0)^c\cap\O}\,:\, u(|x|)\geq t\}=\{x\in {B_T(0)^c\cap\O}\,:\, |x|\leq u^{-1}(t)\}.
$$
For $0=u(+\ty)<t<u(T)$ it is clear that
$$
A(t)=({B_T(0)^c\cap\O})\setminus (B_{u^{-1}(t)}(0)^c\cap\O)=B_{T,u^{-1}(t)}(0)\cap\O.
$$
Furthermore, for $t\geq u(T)$ we have that $A(t)=\emptyset$ because $u(T)=\max_{\tau\geq 0} u(\tau)$ and using \eqref{slojevita_2} we get
$$
\begin{aligned}
\int_{{B_T(0)^c\cap\O}}u(|x|)\di x&=\int_{u(+\ty)}^{u(T)}|B_{T,u^{-1}(t)}(0)\cap\O|\di t=\int_{+\ty}^{T}|B_{T,s}(0)\cap\O|u'(s)\di s\\
&=\int_{T}^{+\ty}|B_{T,s}(0)\cap\O||u'(s)|\di s,
\end{aligned}
$$
where we have introduced the new variable $s=u^{-1}(t)$ in the second to last equality and this concludes the proof of the proposition.
\end{proof}

The following proposition now complements \cite[Prop.\ 4]{ra3}.

\begin{proposition}\label{cor_infint_1}
Let $\O\subseteq\eR^N$ be a measurable set with $|\O|=\ty$, $T>0$ and $\phi>1$.
Then for every $\sigma\in({{\ovb{\dim}}}_B^{\phi}(\ty,\O),+\ty)$, the following identity holds$:$
\begin{equation}\label{infzetaint_inf}
\int_{B_T(0)^c\cap\O}|x|^{-\s-N}\di x=(\s+N)\int_T^{+\ty}t^{-\s-N-1}|B_{T,t}(0)\cap\O|\di t.
\end{equation}
Furthermore, the above integrals are finite for such $\sigma$. 
\end{proposition}

\begin{proof}
First we observe that the condition $|\O|=\ty$ implies that ${{\ovb{\dim}}}_B^{\phi}(\ty,\O)\geq -N$ 
From this we have that for $\sigma\in({{\ovb{\dim}}}_B^{\phi}(\ty,\O),+\ty)$ the function $u(t):=t^{-\s-N}$ satisfies the conditions of Proposition~\ref{integralna_veza_inf} and from that we get~\eqref{infzetaint_inf}.
It remains to show that the integral on the right hand of \eqref{infzetaint_inf} side is finite.
To that end, we fix $\sigma_1\in({{\ovb{\dim}}}_B^{\phi}(\ty,\O),\sigma)$. Then for $T$ large enough we have that for a constant $M>0$ we have
$$
|B_{t,\phi t}\cap\O|\leq Mt^{\sigma_1+N}
$$
for every $t\geq T$, which, in turn, implies that
\begin{equation}\label{B}
|B_{\phi^nT,\phi^{n+1}T}\cap\O|\leq MT^{\sigma_1+N}\phi^{n(\s_1+N)},
\end{equation}
for every $n\in\eN$.
Let us now denote
$$
I_T:=\int_T^{+\ty}t^{-\s-N-1}|B_{T,t}(0)\cap\O|\di t
$$
and calculate
$$
\begin{aligned}
I_T&=\sum_{n=0}^{\ty}\int_{\phi^{n}T}^{\phi^{n+1}T}t^{-\s-N-1}|B_{T,t}(0)\cap\O|\di t\\
&\leq\sum_{n=0}^{\ty}\int_{\phi^{n}T}^{\phi^{n+1}T}t^{-\s-N-1}|B_{T,\phi^{n+1}T}(0)\cap\O|\di t\\
&=\sum_{n=0}^{\ty}\int_{\phi^{n}T}^{\phi^{n+1}T}t^{-\s-N-1}\sum_{k=0}^{n}|B_{\phi^kT,\phi^{k+1}T}(0)\cap\O|\di t.\\
\end{aligned}
$$
Then, by using~\eqref{B}, we have
$$
\begin{aligned}
I_T&\leq MT^{\s_1+N}\sum_{n=0}^{\ty}\sum_{k=0}^{n}\phi^{k(\s_1+N)}\int_{\phi^{n}T}^{\phi^{n+1}T}t^{-\s-N-1}\di t\\
&=MT^{\s_1+N}\sum_{n=0}^{\ty}\frac{\phi^{(n+1)(\s_1+N)}-1}{\phi^{\s_1+N}-1}\,\frac{(\phi^{n}T)^{-\s-N}-(\phi^{n+1}T)^{-\s-N}}{\s+N}\\
&=\underbrace{\frac{MT^{\s_1-\s}(1-\phi^{-\s-N})}{(\s+N)(\phi^{\s_1+N}-1)}}_{K}\sum_{n=0}^{\ty}(\phi^{(n+1)(\s_1+N)}-1)\phi^{n(-\s-N)}\\
&\leq K\sum_{n=0}^{\ty}\phi^{(n+1)(\s_1+N)+n(-\s-N)}=K\phi^{\s_1+N}\sum_{n=0}^{\ty}(\phi^{\s_1-\s})^n<+\ty,
\end{aligned}
$$
since $\s_1-\s<0$ and $\phi>1$.
\end{proof}

Now we can state and prove the holomorphicity theorem for the Lapidus zeta function $\zeta_{\ty,\O}$ at infinity (introduced in \cite[Section 3]{ra2}) that extends \cite[Thm.\ 5]{ra2} and also holds in the case of Lebesgue measurable sets of infinite measure.

\begin{theorem}[Main theorem A]\label{analiticinf_inf}
Let $\O$ be any Lebesgue measurable subset of $\eR^N$.
Assume that $T$ is a fixed positive number and $\phi>1$.
Then the following conclusions hold.

\noindent $(a)$ The distance zeta function at infinity\footnote{Note that $\zeta_{\ty,\O}$ also depends on $T$ but we usually do not denote this because changing $T$ results in adding an entire function to \eqref{inteqzeta_inf}. However, this is not relevant for the singularities of $\zeta_{\ty,\O}$ which are the main topic for the theory of complex dimensions; see \cite[Section 3]{ra2}.}
\begin{equation}\label{inteqzeta_inf}
\zeta_{\ty,\O}(s):=\int_{{B_T(0)^c\cap\O}}|x|^{-s-N}\di x
\end{equation}
is holomorphic on the half-plane $\{\re s>{\ovb{\dim}}_B^{\phi}(\ty,\O)\}$ and for every complex number $s$ in that half-plane
\begin{equation}\label{zetainfder_inf}
\zeta'_{\ty,\O}(s)=-\int_{{B_T(0)^c\cap\O}}|x|^{-s-N}\log|x|\di x.
\end{equation}

\noindent $(b)$ The half-plane from $(a)$ is optimal.\footnote{Optimal in the sense that the integral appearing in~\eqref{inteqzeta_inf} is divergent for $s\in(-\ty,{{\ovb{\dim}}}_B^{\phi}(\ty,\O))$.}

\noindent $(c)$ If $D:={\dim_B^{\phi}(\ty,\O)}$ exists and ${\unb{\M}}_{\phi}^D(\ty,\O)>0$, then $\zeta_{\ty,\O}(s)\to+\ty$ for $s\in\eR$ as $s\to D^+$.
\end{theorem}

\begin{proof}

Firstly we note that if $|\O|<\ty$, then in light of the fact that in this case the upper Minkowski dimension of $\O$ at infinity coincides with its upper $\phi$-shell Minkowski dimension at infinity (see Corollary~\ref{phidim}) the statements $(a)$ and $(b)$ of the theorem follow immediately from \cite[Thm.\ 5]{ra2}.
Additionally, if ${\ovb{\dim}}_B^{\phi}(\ty,\O)<-N$, then part $(c)$ also follows from \cite[Thm.\ 5]{ra2} by using the fact that ${{\unb{\M}_{\phi}^{D}}}(\ty,\O)\leq(1-\phi^{N+D}){\unb{\M}^{D}}(\ty,\O)$ (see Proposition~\ref{prstenasta}), so that ${{\unb{\M}_{\phi}^{D}}}(\ty,\O)>0$ implies that ${\unb{\M}^{D}}(\ty,\O)>0$.

It remains to prove the theorem in the case when $|\O|=\ty$.
First, we fix $\phi>1$ and observe that in light of Proposition~\ref{>=-N} and Proposition~\ref{<=0} we have that ${\ovb{\dim}}_B^{\phi}(\ty,\O)\in[-N,0]$.

$(a)$ If we let $\ovb{D}:={{\ovb{\dim}}}_B^{\phi}(\ty,\O)$, then from the definitions of the upper $\phi$-shell Minkowski content and of the upper $\phi$-shell Minkowski dimension at infinity we have that
$
\limsup_{t\to+\ty}\frac{|B_{t,\phi t}(0)\cap\O|}{t^{N+\sigma}}=0
$
for every $\sigma>\ovb{D}$.
Now, let us fix $\sigma_1$ such that $\ovb{D}<\sigma_1<\sigma$ and take $T>1$ large enough, such that for a constant $M>0$ it holds that
$$
|B_{t,\phi t}(0)\cap\O|\leq Mt^{\sigma_1+N}\quad\textrm{for every}\quad t\geq T,
$$
which implies that
$$
|B_{\phi^nT,\phi^{n+1}T}(0)\cap\O|\leq MT^{\sigma_1+N}\phi^{n(\s_1+N)}\quad\textrm{for every}\quad n\in\eN.
$$
Furthermore, since $\s>\ov{D}\geq -N$ we have that $-\sigma-N<0$ and we estimate $\zeta_{\ty,\O}(\sigma)$ in the following way
$$
\begin{aligned}
\zeta_{\ty,\O}(\sigma)&=\int_{B_T(0)^c\cap\O}|x|^{-\sigma-N}\di x=\sum_{n=0}^{\ty}\int_{B_{\phi^nT,\phi^{n+1}T}\cap\O}|x|^{-\sigma-N}\di x\\
&\leq T^{-\sigma-N}\sum_{n=0}^{\ty}\phi^{n(-\s-N)}|B_{\phi^nT,\phi^{n+1}T}\cap\O|\\
&\leq T^{-\sigma-N}\sum_{n=0}^{\ty}\phi^{n(-\s-N)}MT^{\sigma_1+N}\phi^{n(\s_1+N)}\\
&=MT^{\s_1-\sigma}\sum_{n=0}^{\ty}(\phi^{\s_1-\s})^n<\infty.
\end{aligned}
$$
The last inequality follows from the fact that $\phi>1$ and $\sigma_1-\sigma<0$.
Similarly as in the proof of \cite[Thm.\ 5]{ra2}, we let now $E:={B_T(0)^c\cap\O}$, $\varphi(x):=|x|$ and $\di\mu(x):=|x|^{-N}\di x$ and note that $\varphi(x)\geq T>1$ for $x\in E$.
Part $(a)$ follows now from \cite[Thm.\ 4(b)]{ra2}.

To prove part $(b)$ of the theorem we denote $\ovb{D}:={\ovb{\dim}}_B^{\phi}(\ty,\O)\in[-N,0]$.
In case $s\leq -N$ we have
$$
\int_{B_T(0)^c\cap\O}|x|^{-s-N}\di x\geq T^{-s-N}\int_{B_T(0)^c\cap\O}\di x=+\ty.
$$
It remains to prove the case when $s\in(-N,\ovb{D})$.
By using Proposition~\ref{cor_infint_1}, we have
\begin{equation}
\begin{aligned}
I_T:=\int_{B_T(0)^c\cap\O}|x|^{-s-N}\di x&=(s+N)\int_{T}^{+\ty}t^{-s-N-1}|B_{T,t}(0)\cap\O|\di t\\
&\geq (s+N)\int_{\phi T}^{\phi^2 T}t^{-s-N-1}|B_{T,t}(0)\cap\O|\di t\\
&\geq (s+N)|B_{T,\phi T}(0)\cap\O|\int_{\phi T}^{\phi^2 T}t^{-s-N-1}\di t\\
&=\phi^{-s-N}(1-\phi^{-s-N})T^{-s-N}|B_{T,\phi T}(0)\cap\O|.
\end{aligned}
\end{equation}
Now, we fix $\sigma$ such that $s<\sigma<\ovb{D}$.
From ${{\ovb{\M}}}_{\phi}^{\sigma}(\ty,\O)=+\ty$ we conclude that there exists a sequence $(t_k)_{k\geq 1}$ such that
$$
C_k:=\frac{|{B_{t_k,\phi t_k}(0)\cap\O}|}{t_k^{N+\sigma}}\to+\ty\quad\textrm{when}\quad t_k\to+\ty.
$$
It is clear that the function $T\to I_T$ is nonincreasing and we have
\begin{equation}
\begin{aligned}
I_T\geq I_{t_k}&\geq \phi^{-s-N}(1-\phi^{-s-N})t_k^{-s-N}|B_{t_k,\phi t_k}(0)\cap\O|\\
&=\phi^{-s-N}(1-\phi^{-s-N})t_k^{-s-N}C_kt_k^{N+\s}\\
&=\phi^{-s-N}(1-\phi^{-s-N})C_kt_k^{\sigma-s}\to+\ty.
\end{aligned}
\end{equation}
Therefore, $I_T=+\ty$ for every $s<\ovb{D}$ which proves part $(b)$.

For part $(c)$ we assume that $D={\dim}_B^{\phi}(\ty,\O)$ exists, $D\in[-N,0]$ and ${{\unb{\M}}}_{\phi}^D(\ty,\O)>0$.
This implies that there exists a constant $C>0$ such that for a sufficiently large $T$ we have that $|B_{t,\phi t}(0)\cap\O|\geq Ct^{N+D}$ for every $t\geq T$.
Now, for $-N<D<s$ we have the following:
\begin{equation}\label{racun_ty}
\begin{aligned}
\zeta_{\ty,\O}(s)&=\int_{B_T(0)^c\cap\O}|x|^{-s-N}\di x=(s+N)\int_T^{+\ty}t^{-s-N-1}|B_{T,t}(0)\cap\O|\di t\\
&\geq (s+N)\int_{\phi T}^{+\ty}t^{-s-N-1}|B_{T,t}(0)\cap\O|\di t\\
&\geq (s+N)\int_{\phi T}^{+\ty}t^{-s-N-1}|B_{\phi^{-1}t,t}(0)\cap\O|\di t\\
&\geq C(s+N)\int_{\phi T}^{+\ty}t^{-s-N-1}(\phi^{-1}t)^{N+D}\di t\\
&=C(s+N)\phi^{-N-D}\frac{T^{D-s}}{s-D}\to+\ty,
\end{aligned} 
\end{equation}
when $s\to D^+$, and this proves part $(c)$ in the case when $D\in(-N,0]$.
For the special case when $D=-N<s$ we proceed in a slightly different manner:
$$
\begin{aligned}
\zeta_{\ty,\O}(s)&=\int_{B_T(0)^c\cap\O}|x|^{D-s}\di x=(s-D)\int_T^{+\ty}t^{D-s-1}|B_{T,t}(0)\cap\O|\di t\\
&\geq (s-D)\int_{T}^{+\ty}t^{D-s-1}|B_{T,\phi^{\lfloor\log_{\phi}(t/T)\rfloor}T}(0)\cap\O|\di t\\
&\geq (s-D)\int_{T}^{+\ty}t^{D-s-1}\sum_{k=0}^{\lfloor\log_{\phi}(t/T)\rfloor-1}|B_{\phi^kT,\phi^{k+1}T}(0)\cap\O|\di t\\
&\geq C(s-D)\int_{T}^{+\ty}t^{D-s-1}\lfloor\log_{\phi}(t/T)\rfloor\di t\\
&\geq C(s-D)\int_{T}^{+\ty}t^{D-s-1}(\log_{\phi}t-\log_{\phi}T-1)\di t=:I.
\end{aligned} 
$$
The last integral appearing above (denoted by $I$) can be explicitly calculated:
$$
\begin{aligned}
I&=\frac{C(s-D)}{\log\phi}\left(\frac{t^{D-s}\log t}{D-s}-\frac{t^{D-s}}{(D-s)^2}\right)\Bigg|_{T}^{+\ty}-C(s-D)\frac{(\log_{\phi}T+1)t^{D-s}}{D-s}\Bigg|_{T}^{+\ty}\\
&=C(s-D)\left(\frac{T^{D-s}}{(s-D)^2\log{\phi}}+\frac{T^{D-s}\log_{\phi}T}{s-D}-\frac{T^{D-s}(\log_{\phi}T+1)}{s-D}\right)\\
&=\frac{C}{\log\phi}\,\frac{T^{D-s}}{s-D}-CT^{D-s}\to +\ty,
\end{aligned} 
$$
when $s\to D^+$, which concludes the proof of the theorem.
\end{proof}

%

The next theorem generalizes \cite[Thm.\ 2]{ra3} for the case when we do not require that $\O$ is of finite Lebesgue measure.

\begin{theorem}[Functional equation for the zeta function at infinity]\label{tubedistinf_inf}
Let $\O$ be a Lebesgue measurable subset of $\eR^N$, $T>0$ fixed and $\phi>1$ fixed.
Then, the functional equation
\begin{equation}\label{tubedistinfeq_inf}
\zeta_{\ty,\Omega}(s;T)=(s+N)\displaystyle\int_T^{+\ty}t^{-s-N-1}|B_{T,t}(0)\cap\O|\di t,
\end{equation}
is valid for all $s\in\Ce$ such that $\re s>\max\{-N,{{\ovb{\dim}}}_B^{\phi}(\ty,\O)\}$, i.e., the integral on the right hand side is absolutely convergent in the open right half-plane $\Pi:=\{\re s>\max\{-N,{{\ovb{\dim}}}_B^{\phi}(\ty,\O)\}\}$ and defines a holomorphic function in that domain.

%
\end{theorem}

\begin{proof}
First, we will show that in the case when $|\O|<\ty$, this is actually a rewriting of \cite[Thm.\ 2]{ra3}.
Namely, $|\O|<\ty$ implies that ${{\ovb{\dim}}}_B^{\phi}(\ty,\O)={{\ovb{\dim}}}_B(\ty,\O)\leq-N$ for every $\phi>1$ (see \cite[Prop.\ 1]{ra2}).
Here we introduce a short-hand notation: ${_t\O}:=B_t(0)^c\cap\O$ for any $t>0$ for the sake of brevity. 
Then, for $\re s>-N$ we note that $-\re s-N<0$ and from \cite[Thm.\ 2]{ra3}, we have
$$
\begin{aligned}
\int_{{_T\O}}&|x|^{-s-N}\di x=T^{-s-N}|{_T\O}|-(s+N)\displaystyle\int_T^{+\ty}t^{-s-N-1}|_t\O|\di t\\
&=T^{-s-N}|{_T\O}|-(s+N)\int_T^{+\ty}t^{-s-N-1}(|{_T\O}|-|B_{T,t}(0)\cap\O|)\di t\\
&=T^{-s-N}|{_T\O}|-|{_T\O}|t^{-s-N}\bigg|_{T}^{+\ty}+(s+N)\int_T^{+\ty}t^{-s-N-1}|B_{T,t}(0)\cap\O|\di t\\
&=(s+N)\int_T^{+\ty}t^{-s-N-1}|B_{T,t}(0)\cap\O|\di t.
\end{aligned}
$$
Again, from \cite[Thm.\ 2]{ra3} we have that the left-hand side above is holomorphic in $\{\re s>-N\}$.
On the other hand, we have that
$$
\left|\int_T^{+\ty}t^{-s-N-1}|B_{T,t}(0)\cap\O|\di t\right|\leq|\O|\int_T^{+\ty}t^{-\re s-N-1}\di t=\frac{|\O|T^{-\re s-N}}{\re s+N}
$$
and by a classical result the above integral defines a holomorphic function on $\{\re s>-N\}$.

It remains to prove the theorem in the case when $|\O|=\ty$ and, in light of Proposition~\ref{>=-N}, we then have that ${{\ovb{\dim}}}_B^{\phi}(\ty,\O)\geq -N$ for every $\phi>1$.

From Proposition~\ref{cor_infint_1} we have that~\eqref{tubedistinfeq_inf} is valid for $\eR\ni\s>{{\ovb{\dim}}}_B^{\phi}(\ty,\O)$ and both integrals are finite.
Furthermore, to show that the equality holds in the half-plane $\{\re s>{{\ovb{\dim}}}_B^{\phi}(\ty,\O)\}$, it suffices to prove that both sides of Equation~\eqref{tubedistinfeq_inf} are holomorphic functions on that domain.\footnote{The equality follows from the fact that two holomorphic functions that coincide on a set that has an accumulation point in their common domain coincide then on the whole common domain.}
We already have that the left-hand side of~\eqref{tubedistinfeq_inf} is holomorphic on the set $\{\re s>{{\ovb{\dim}}}_B^{\phi}(\ty,\O)\}$ according to Theorem~\ref{analiticinf_inf}.
Furthermore, the right-hand side of~\eqref{tubedistinfeq_inf} is a Dirichlet type integral with $\varphi(t)=t^{-s}$ and $\di\mu(t)=t^{-N-1}|B_{T,t}(0)\cap\O|\di t$, and according to \cite[Thm.\ 4(b)]{ra2} it is sufficient to show that the integral on the right hand side of~\eqref{tubedistinfeq_inf} is convergent for $\re s>{{\ovb{\dim}}}_B^{\phi}(\ty,\O)$.

For $\ovb{D}:={{\ovb{\dim}}}_B^{\phi}(\ty,\O)$ and $s\in\Ce$ such that $\re s>\ovb{D}$, let us choose $\e>0$ sufficiently small such that $\re s>\ovb{D}+\e$.
Since ${{\ovb{\M}}}_{\phi}^{\ovb{D}+\e}(\ty,\O)=0$, there exists a constant $C_T>0$ such that $|B_{t,\phi t}(0)\cap\O|\leq C_Tt^{N+\ovb{D}+\e}$ for every $t\in[T,+\ty)$. 
Now we have the following estimate exactly in the same way as in the proof of the second part of Proposition~\ref{cor_infint_1} (by letting $\s=\re s$ and $\s_1=\ovb{D}+\e$ in the notation of that proof, and $K$ being a positive constant):
\begin{equation}
\begin{aligned}
\left|\int_{T}^{+\ty}t^{-s-N-1}|B_{T,t}(0)\cap\O|\di t\right|&\leq\int_{T}^{+\ty}t^{-\re s-N-1}|B_{T,t}(0)\cap\O|\di t\\
&\leq K\phi^{\s_1+N}\sum_{n=0}^{\ty}(\phi^{\ovb{D}+\e-\re s})^n<+\ty.
\end{aligned}
\end{equation}
This, together with the principle of analytic continuation, completes the proof of the theorem.
\end{proof}

Now we state our second main theorem that expands \cite[Thm.\ 6]{ra2} to the case of unbounded sets of infinite Lebesgue measure.

\begin{theorem}[Main theorem B]\label{resdistinf_inf}
Let $\O$ be a Lebesgue measurable set and $\phi>1$ such that ${\dim_B^{\phi}}(\ty,\O)=D>-\ty$ exists.
Furthermore, let $0<{{\unb{\M}}}_{\phi}^D(\ty,\O)\leq{{\ovb{\M}}}_{\phi}^D(\ty,\O)<\ty$.
If $\zeta_{\ty,\O}$ has a meromorphic continuation to a neighborhood of $s=D$, then $D$ is a simple pole.
Furthermore in the case when $D\in[-N,0]$ we have that
\begin{equation}\label{mink_res_inf_inf}
\frac{1}{\phi^{N+D}\log\phi}{{\unb{\M}}}_{\phi}^D(\ty,\O)\leq\res(\zeta_{\ty,\O},D)\leq\frac{1}{\log\phi}{{\ovb{\M}}}_{\phi}^D(\ty,\O) 
\end{equation}
and in the case when $D\in(-\ty,-N)$ we have that
\begin{equation}\label{mink_res_inf_<inf}
-\frac{N+D}{1-\phi^{N+D}}{{\unb{\M}}}_{\phi}^D(\ty,\O)\leq\res(\zeta_{\ty,\O},D)\leq-\frac{N+D}{1-\phi^{N+D}}{{\ovb{\M}}}_{\phi}^D(\ty,\O).
\end{equation}

In addition, if we assume that $\O$ is $\psi$-shell Minkowski measurable at infinity for every $\psi\in(1,\phi)$, we have that
\begin{equation}\label{res_lim}
\res(\zeta_{\ty,\O},D)=\lim_{\psi\to 1^+}\frac{{{{\M}}}_{\psi}^D(\ty,\O)}{\log\psi}.
\end{equation}

\end{theorem}

\begin{proof}
Firstly, by looking at the proof of part $(c)$ of Theorem~\ref{analiticinf_inf} we can conclude that $s=D$ is a singularity of $\zeta_{\ty,\O}$ which is at least a simple pole.
It remains to prove that the order of this pole is not larger than one.
We let
$$
C_T:=\sup_{t\geq T}\frac{|B_{t,\phi t}(0)\cap\O|}{t^{N+D}}
$$
and conclude from ${{\ovb{\M}}}_{\phi}^D(\ty,\O)<+\ty$ that we have $C_T<+\ty$ for $T$ large enough.
Let us first assume that $-N<D\leq 0$ and take $s\in\eR$ such that $s>D$.
From Theorem~\ref{tubedistinf_inf} we then have
$$
\begin{aligned}
\zeta_{\ty,\Omega}(s;T)&=(s+N)\int_{T}^{+\ty}t^{-s-N-1}|B_{T,t}(0)\cap\O|\di t\\
&=(s+N)\sum_{n=0}^{\ty}\int_{\phi^nT}^{\phi^{n+1}T}t^{-s-N-1}|B_{T,t}(0)\cap\O|\di t\\
&\leq(s+N)\sum_{n=0}^{\ty}\int_{\phi^nT}^{\phi^{n+1}T}t^{-s-N-1}|B_{T,\phi^{n+1}T}(0)\cap\O|\di t.\\
\end{aligned}
$$
Furthermore, from the definition of $C_T$ we have
$$
\begin{aligned}
\zeta_{\ty,\Omega}(s;T)&\leq(s+N)\sum_{n=0}^{\ty}\int_{\phi^nT}^{\phi^{n+1}T}t^{-s-N-1}\sum_{k=0}^{n}|B_{\phi^{k}T,\phi^{k+1}T}(0)\cap\O|\di t\\
&\leq (s+N)C_TT^{N+D}\sum_{n=0}^{\ty}\sum_{k=0}^{n}\phi^{k(N+D)}\int_{\phi^nT}^{\phi^{n+1}T}t^{-s-N-1}\di t\\
&=(s+N)C_TT^{N+D}\sum_{n=0}^{\ty}\frac{\phi^{(n+1)(N+D)}-1}{\phi^{N+D}-1}\,\frac{(\phi^{n}T)^{-s-N}-(\phi^{n+1}T)^{-s-N}}{s+N}\\
&=\frac{C_TT^{D-s}(1-\phi^{-s-N})}{\phi^{N+D}-1}\sum_{n=0}^{\ty}(\phi^{(n+1)(N+D)}-1)\phi^{n(-s-N)}\\
\end{aligned}
$$
which can be further bounded by neglecting the $-1$ from the braces above to get 
$$
\begin{aligned}
\zeta_{\ty,\Omega}(s;T)&\leq\frac{C_TT^{D-s}(1-\phi^{-s-N})}{\phi^{N+D}-1}\sum_{n=0}^{\ty}\phi^{(n+1)(N+D)+n(-s-N)}\\
&=\frac{C_TT^{D-s}(1-\phi^{-s-N})\phi^{N+D}}{\phi^{N+D}-1}\sum_{n=0}^{\ty}\phi^{n(D-s)}=\frac{C_TT^{D-s}(\phi^{N+D}-\phi^{D-s})}{(\phi^{N+D}-1)(1-\phi^{D-s})}.
\end{aligned}
$$
From this we conclude that $|\zeta_{\ty,\O}(s;T)|\leq C(1-\phi^{D-s})^{-1}$ where $C>0$ is a positive constant independent of $s$ and $T$ which implies that $s=D$ is a pole of order at most one, i.e., a simple pole.
Since changing $T$ amounts to adding an entire function to $\zeta_{\ty,\O}(\,\cdot\,;T)$; see \cite[Section 3]{ra2}, we conclude that the residue at $s=D$ of $\zeta_{\ty,\O}(\,\cdot\,;T)$ is independent of $T$ and for $s>D$ we have that
$$
(s-D)\zeta_{\ty,\O}(s)\leq\frac{C_TT^{D-s}(\phi^{N+D}-\phi^{D-s})}{\phi^{N+D}-1}\,\frac{s-D}{1-\phi^{D-s}}
$$
which, by letting $s\to D^+$, yields
$$
\res(\zeta_{\ty,\O},D)\leq\frac{C_T}{\log\phi}.
$$
Finally, by taking the limit as $T\to\ +\ty$ we get
$$
\res(\zeta_{\ty,\O},D)\leq\frac{1}{\log\phi}{{\ovb{\M}}}_{\phi}^D(\ty,\O).
$$
For the inequality involving the lower $\phi$-shell Minkowski content at infinity, we define
$$
K_T:=\inf_{t\geq T}\frac{|B_{t,\phi t}(0)\cap\O|}{t^{N+D}}
$$
and conclude from ${{\unb{\M}}}_{\phi}^D(\ty,\O)>0$ that we have $K_T>0$ for $T$ large enough.
Furthermore, we take $s>D$ and proceed in a similar manner as before:
$$
\begin{aligned}
\zeta_{\ty,\Omega}(s;T)&=(s+N)\sum_{n=0}^{\ty}\int_{\phi^nT}^{\phi^{n+1}T}t^{-s-N-1}|B_{T,t}(0)\cap\O|\di t\\
&\geq(s+N)\sum_{n=0}^{\ty}\int_{\phi^nT}^{\phi^{n+1}T}t^{-s-N-1}|B_{T,\phi^{n}T}(0)\cap\O|\di t\\
&=(s+N)\sum_{n=0}^{\ty}\int_{\phi^nT}^{\phi^{n+1}T}t^{-s-N-1}\sum_{k=0}^{n-1}|B_{\phi^{k}T,\phi^{k+1}T}(0)\cap\O|\di t.\\
\end{aligned}
$$
Similarly as before, from the definition of $K_T$ we get:
$$
\begin{aligned}
\zeta_{\ty,\Omega}(s;T)&\geq (s+N)K_TT^{N+D}\sum_{n=0}^{\ty}\sum_{k=0}^{n-1}\phi^{k(N+D)}\int_{\phi^nT}^{\phi^{n+1}T}t^{-s-N-1}\di t\\
&=(s+N)K_TT^{N+D}\sum_{n=0}^{\ty}\frac{\phi^{n(N+D)}-1}{\phi^{N+D}-1}\,\frac{(\phi^{n}T)^{-s-N}-(\phi^{n+1}T)^{-s-N}}{s+N}\\
&=\frac{K_TT^{D-s}(1-\phi^{-s-N})}{\phi^{N+D}-1}\sum_{n=0}^{\ty}(\phi^{n(N+D)}-1)\phi^{n(-s-N)}.\\
\end{aligned}
$$
Interchanging summation and subtraction above yields
$$
\begin{aligned}
\zeta_{\ty,\Omega}(s;T)&\geq\frac{K_TT^{D-s}(1-\phi^{-s-N})}{\phi^{N+D}-1}\left(\sum_{n=0}^{\ty}\phi^{n(D-s)}-\frac{1}{1-\phi^{-s-N}}\right)\\
&=\frac{K_TT^{D-s}(1-\phi^{-s-N})}{(\phi^{N+D}-1)(1-\phi^{D-s})}-\frac{K_TT^{D-s}}{(\phi^{N+D}-1)}.
\end{aligned}
$$
This implies that
$$
(s-D)\zeta_{\ty,\O}(s)\geq \frac{K_TT^{D-s}(1-\phi^{-s-N})}{\phi^{N+D}-1}\,\frac{s-D}{1-\phi^{D-s}}-\frac{K_TT^{D-s}(s-D)}{(\phi^{N+D}-1)},
$$
and by letting $s\to D^+$ we have
$
\res(\zeta_{\ty,\O})\geq\frac{K_T}{\phi^{N+D}\log\phi}.
$
Finally, we let $T\to +\ty$ to get
$
\res(\zeta_{\ty,\O},D)\geq\frac{1}{\phi^{N+D}\log\phi}{{\unb{\M}}}_{\phi}^D(\ty,\O).
$

Let us now treat the special case when $D=-N$.
We take $s\in\eR$ such that $s>D$ and, similarly as before, we have
\begin{equation}\nonumber
\begin{aligned}
\zeta_{\ty,\O}(s)&\leq (s+N)\int_T^{+\ty}t^{-s-N-1}\sum_{n=0}^{\lfloor\log_\phi(t/T)\rfloor}|B_{\phi^nT,\phi^{n+1}T}\cap\O|\di t\\
&\leq C_T(s+N)\int_T^{+\ty}t^{-s-N-1}\sum_{n=0}^{\lfloor\log_\phi(t/T)\rfloor}\phi^{n(N+D)}\di t\\
&= C_T(s+N)\int_T^{+\ty}t^{-s-N-1}(\lfloor\log_\phi(t/T)\rfloor+1)\di t\\
&\leq \frac{C_T(s+N)}{\log\phi}\int_T^{+\ty}t^{-s-N-1}\log t\di t=\frac{C_TT^{-s-N}}{(s+N)\log\phi}(N\log T+s\log T+1).
\end{aligned}
\end{equation}
From this, we conclude that
$$
\res(\zeta_{\ty,\O},-N)\leq \frac{C_T}{\log\phi},
$$
and, by letting $T\to +\ty$ we get the desired inequality.
For the other inequality we have the following estimates:
\begin{equation}\nonumber
\begin{aligned}
\zeta_{\ty,\O}(s)&\geq (s+N)\int_T^{+\ty}t^{-s-N-1}\sum_{n=0}^{\lfloor\log_\phi(t/T)\rfloor-1}|B_{\phi^nT,\phi^{n+1}T}\cap\O|\di t\\
&\geq K_T(s+N)\int_T^{+\ty}t^{-s-N-1}\sum_{n=0}^{\lfloor\log_\phi(t/T)\rfloor-1}\phi^{n(N+D)}\di t\\
&\geq K_T(s+N)\int_T^{+\ty}t^{-s-N-1}(\log_\phi(t/T)-1)\di t\\
&\geq \frac{K_T(s+N)}{\log\phi}\int_T^{+\ty}t^{-s-N-1}\log t\di t-K_T(s+N)(1+\log_\phi T)\int_T^{+\ty}t^{-s-N-1}\di t\\
&=\frac{K_TT^{-s-N}}{(s+N)\log\phi}(N\log T+s\log T+1)-K_T(1+\log_\phi T){T^{-s-N}}.
\end{aligned}
\end{equation}
Finally, as before, we first multiply both sides by $(s+N)$, let $s\to -N^+$ and then let $T\to +\ty$.

It remains to prove the theorem in the case when $D\in(-\ty,-N)$.
The argumentation will be similar as before but we will assume that $D<s<-N$ and use~\eqref{tubedistinfeq_inf} to get
\begin{equation}\nonumber
\begin{aligned}
\zeta_{\ty,\O}(s)&\leq T^{-s-N}|B_T(0)^c\cap\O|-(s+N)\int_T^{+\ty}t^{-s-N-1}|B_{t}(0)^c\cap\O|\di t\\
&\leq -(s+N)\int_T^{+\ty}t^{-s-N-1}\sum_{n=0}^{\ty}|B_{\phi^nt,\phi^{n+1}t}(0)\cap\O|\di t\\
&\leq -C_T(s+N)\int_T^{+\ty}t^{-s-N-1}\sum_{n=0}^{\ty}t^{N+D}\phi^{n(N+D)}\di t\\
&= -\frac{C_T(s+N)}{1-\phi^{N+D}}\int_T^{+\ty}t^{D-s-1}\di t=-\frac{C_T(s+N)}{1-\phi^{N+D}}\,\frac{T^{D-s}}{s-D}.
\end{aligned}
\end{equation}
Exactly as before, we multiply both sides by $(s-D)$, let $s\to D^+$ and, after that, let $T\to+\ty$.
To get the other inequality, and conclude the proof, we proceed in a similar manner by using the lower $\phi$-shell Minkowski content of $\O$ at infinity.  

It remains to prove the last assertion of the theorem when we assume additionally that $\O$ is $\psi$-shell Minkowski measurable for every $\psi\in(0,\phi)$.
Note that \eqref{mink_res_inf_inf}  implies that the limit $\lim_{\psi\to 1^+}{{{{\M}}}_{\psi}^D(\ty,\O)}/{\log\psi}$ exists and this resolves the case when $D\in[-N,0]$.
It remains only to see that~\eqref{res_lim} holds even in the case when $D\in(-\ty,-N)$, but this is a simple consequence of L'Hospital's rule and Equation \eqref{mink_res_inf_<inf}:
$$
\begin{aligned}
\lim_{\psi\to 1^+}\frac{{{{\M}}}_{\psi}^D(\ty,\O)}{\log\psi}&=\lim_{\psi\to 1^+}\frac{1-\psi^{N+D}}{-(N+D)}\,\frac{\res(\zeta_{\ty,\O},D)}{\log\psi}\\&=\res(\zeta_{\ty,\O},D)\lim_{\psi\to 1^+}\frac{-(N+D)\psi^{N+D-1}}{-(N+D)\psi^{-1}}=\res(\zeta_{\ty,\O},D).\\
\end{aligned}
$$ 
\end{proof}

\section{Quasiperiodic sets at infinity of infinite Lebesgue measure}
\label{sec:quasi}

In this section we construct quasiperiodic sets at infinity, similarly as it was done in \cite{ra3} but, here they will have infinite Lebesgue measure and hence we will use the $\phi$-shell Minkowski dimension in order to describe their fractality.

We recall the two parameter set $\Omega_{\ty}^{(a,b)}$ from \cite[Section 5]{ra2}.
\begin{definition}\label{Omega(a,b,ty)}
For $a\in(0,1/2)$ and $b\in(1+\log_{1/a}2,+\ty)$ we 
define a countable family of sets
$$
\O_{m}^{(a,b)}:=\{(x,y)\in\eR^2\,:\,x>a^{-m},\ 0<y<x^{-b}\},\quad m\geq 1.
$$
The set $\O_{\ty}^{(a,b)}$ is then constructed by ``stacking'' the translated images of the sets $\O_{m}^{(a,b)}$ along the $y$-axis on top of each other.
More precisely, for each $m\geq 1$ we take $2^{m-1}$ copies of $\O_{m}^{(a,b)}$ and arrange all of these sets by vertical translations so that they are pairwise disjoint and lie in the strip $\{0\leq y\leq S\}$ where
$S$ is the sum of all the widths of all the sets in the union, i.e., 
$
S=\sum_{m=1}^{\ty}2^{m-1}\cdot(a^{-m})^{-b}=\frac{a^{b}}{1-2a^{b}};
$
see Figure~\ref{Cantor_u_besk}.
Hence, $\O_{\ty}^{(a,b)}$ is the disjoint union of all of these sets.
\end{definition}

Note that in the above definition, the assumption on the parameter $a$ guarantees that $\O_{\ty}^{(a,b)}$ lies in a strip of finite height while the assumption on the parameter $b$ guarantees that $\O_{\ty}^{(a,b)}$ has finite Lebesgue measure.

We will now remove the restriction on $b$ and let $b\in(\log_{1/a}2,1+\log_{1/a}2]$ so that the set $\Omega_{\ty}^{(a,b)}$ will be of infinite Lebesgue measure \cite[Rem.\ 6]{ra2} but it will still be contained in a strip $\{0\leq y\leq S\}$ of finite height $S=\frac{a^b}{1-2a^b}$.
By going carefully through \cite[Example 6]{ra2} one can see that the condition on $\Omega_{\ty}^{(a,b)}$ being of finite Lebesgue measure is not really needed and thus we obtain the following result which generalizes Example \cite[Example 6]{ra2} to sets of infinite Lebesgue measure.

\begin{example}\label{tw-param-inf}
Let $\O_{\ty}^{(a,b)}$ be the two parameter set from \cite[Section 5]{ra2}; see Figure \ref{Cantor_u_besk}, but of infinite Lebesgue measure; that is, with $a\in(0,1/2)$ and $b\in(\log_{1/a}2,1+\log_{1/a}2]$.
Then, its distance zeta function at infinity calculated via the $|\cdot|_{\ty}$-norm on $\eR^2$ is given by
\begin{equation}
\zeta_{\ty,\O_{\ty}^{(a,b)}}(s;|\cdot|_{\ty})=\frac{1}{s+b+1}\cdot\frac{1}{a^{-(s+b+1)}-2}.
\end{equation}
It is meromorphic on $\Ce$ where the set of complex dimensions of $\O_{\ty}^{(a,b)}$ at infinity visible through $W:=\{\re s>\log_{1/a}-b-3\}$ is given by\footnote{We define the complex dimensions of sets at infinity with infinite Lebesgue measure in the same way as for sets of finite Lebesgue measure.}
\begin{equation}\label{twop_set_ty}
\{-(b+1)\}\cup\left(\log_{1/a}2-(b+1)+\frac{2\pi}{\log (1/a)}\I\Ze\right).
\end{equation}
Furthermore, we also have that for any $\phi>1$,
\begin{equation}\label{fidim}
{{\ovb\dim}_{B}^{\phi}}(\ty,\O_{\ty}^{(a,b)})=\log_{1/a}2-(b+1).
\end{equation}
\end{example}

We have already discussed the first part of the example. Furthermore, by going carefully through \cite[Def.\ 1 and Thm.\ 11]{ra3} one can see that the assumption of $\O$ having finite Lebesgue measure is not really needed there. 
Hence, we conclude that \eqref{twop_set_ty} holds in the same way as in \cite[Section 5]{ra2}.
Finally, \eqref{fidim} follows from the fact that the upper $\phi$-shell Minkowski dimension does not depend on the choice of $\phi$ (see Proposition \ref{phi_independent}) and from Theorem \ref{analiticinf_inf} part $(b)$.

\begin{figure}[h]
\begin{center}
\includegraphics[width=10cm,height=5cm]{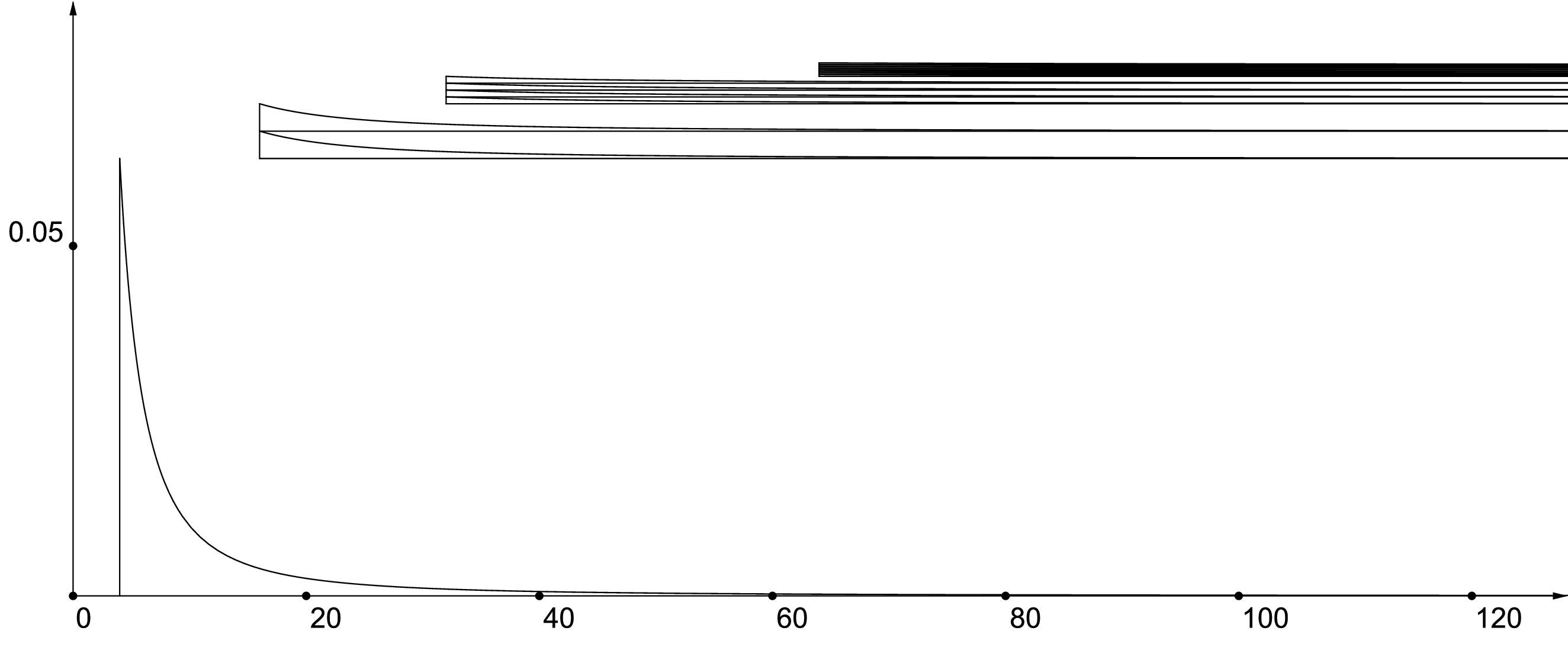}
\end{center}
\caption{The two parameter set $\O_{\ty}^{(a,b)}$ with the first four steps of the construction of the set $\O_{\ty}^{(a,b)}$ depicted. The idea of the construction is to stack copies of the sets $
\O_{m}^{(a,b)}:=\{(x,y)\in\eR^2\,:\,x>a^{-m},\ 0<y<x^{-b}\},\ m\in\eN
$ where for each $m\in\eN$ we stack $2^{m-1}$ copies of $
\O_{m}^{(a,b)}$; see \cite[Section 5]{ra2}.}
\label{Cantor_u_besk}
\end{figure}

One can now construct quasiperiodic sets of infinite Lebesgue measure at infinity by using the two parameter sets from Example \ref{tw-param-inf} as building blocks analogously as was done in \cite[Section 5]{ra3} to construct them in the case of finite Lebesgue measure. The same is true for the maximally hyperfractal case. Here we only briefly summarize the results.

Recall that we call an unbounded set of finite Lebesgue measure {\em maximally hyperfractal at infinity} if its distance zeta function at infinity has the critical line $\{\re s=\ov{\dim}_B(\infty,\O)\}$ as its natural boundary \cite{ra3}.
We generalize this definition for any unbounded set by replacing $\ov{\dim}_B(\infty,\O)$ by $\ov{\dim}^{\phi}_B(\infty,\O)$ for some arbitrary $\phi >1$ (the upper $\phi$-shell Minkowski dimension at infinity does not depend on $\phi$).

\begin{theorem}[Existence of maximally hyperfractal sets at infinity]
 For any $D \leq 0$, there exists a set $\O\subseteq\eR^2$ which is maximally hyperfractal at infinity with $\ov{\dim}^{\phi}_B(\infty,\O)=D$ and is Minkowski nondegenerate at infinity.	
\end{theorem}

By adapting the definition of a quasiperiodic set at infinity as follows (compare to \cite[Def.\ 2]{ra3}) we can also construct algebraically and transcendentally quasiperiodic sets at infinity of any dimension $D\leq 0$.

\begin{definition}[General quasiperiodic set at infinity]\label{quasiperiodic}
Let $\O\subseteq\eR^N$ be Lebesgue measurable (possibly of infinite Lebesgue measure) and such that for some fixed  $D\leq 0$ it satisfies the following asymptotics:\footnote{Under these conditions it follows directly that then $\dim_B^{\phi}(\ty,\O)=D$, $\unb{\M}^D_{\phi}(\ty,\O)=\phi^N\liminf_{\tau\to+\ty} G(\tau)$ and $\ovb{\M}^{D}_{\phi}(\ty,\O)=\phi^N\limsup_{\tau\to+\ty} G(\tau)$ for any $\phi>1$.}
\begin{equation}\label{Aasymp}
|B_{T,t}(0)\cap\O| =t^{N+D}(G(\log t)+o(1))\q\textrm{as}\q t\to +\ty,
\end{equation}
where $G(\tau)\geq0$ for all $\tau\in\mathbb{R}$ and satisfies $$0<\liminf_{\tau\to+\ty} G(\tau)\le\limsup_{\tau\to+\ty} G(\tau)<+\ty.$$

We then call $\O$ an {\em algebraically or transcendentally $n$-quasiperiodic set}\label{quasiperiodicdef}
{\em at infinity} if $G=G(\tau)$ is an algebraically or transcendentally $n$-quasiperiodic function, with $n\in \{2,3,\ldots\}\cup\{\infty\}$, respectively, in the following sense \cite[Def.\ 4.6.6]{fzf}.

A function $G=G(\tau)\colon\eR\to\eR$ is said to be {\em $n$-quasiperiodic} with $n<\infty$ if it is of the form
$$
G(\tau)=H(\tau,\ldots,\tau),
$$
where for some $n\geq 2$, $H\colon\eR^{n}\to\eR$ is a nonconstant $T_k$-periodic function in its $k$-th component, for each $k=1,\ldots,n$, and the corresponding periods $T_1,\ldots,T_n$ are rationally independent.
The values $T_k$ are called the {\em quasiperiods of $G$}.

In addition, we say that a function $G=G(\tau)$ is

\smallskip

$(a)$ {\em transcendentally $n$-quasiperiodic} if the
periods $T_1,\dots, T_n$ are algebraically independent, i.e., 
all of the quotients $T_i/T_j$, for $i\ne j$, are transcendental numbers.

$(b)$ {\em algebraically $n$-quasiperiodic}
if the corresponding periods $T_1,\dots, T_n$ are algebraically dependent, i.e., there exist algebraic numbers $\g_1,\dots,\g_n$, not all of them zero, such that $\g_1 T_1+\dots+\g_n T_n=0$.

Finally, a function $G:\eR\to\eR$ is said to be {\em $\ty$-quasiperiodic},
if it is of the form
$$
G(\tau)=H(\tau,\tau,\dots),
$$
where $H:\eR_b^\ty\to\eR$,\footnote{$\eR_b^\ty$ stands here for the usual Banach space\label{banach_bdd} of bounded sequences $(\tau_j)_{j\ge1}$ of real numbers, endowed with
the norm $\|(\tau_j)_{j\ge1}\|_\ty:=\sup_{j\ge1}|\tau_j|$.}
 $H=H(\tau_1,\tau_2,\dots)$ is a function which is $T_j$-periodic in its $j$-th component, for each $j\in\eN$,
with $T_j>0$ as minimal periods, and such that the set of periods 
$
\{T_j:j\ge1\}
$
is {\em rationally} independent.
We make the same distinction between transcendental and algebraic $\infty$-quasiperiodicity as in the case of finite number of quasiperiods.

%
%
%
%
We denote the families of algebraically and transcendentally $n$-quasiperiodic sets at infinity, by $\mathscr{D}_{aqp}^{\ty}(n)$ and $\mathscr{D}_{tqp}^{\ty}(n)$, respectively. 
\end{definition}

It is easy to check that the above generalized definition is compatible with \cite[Def.\ 2]{ra3} when $\Omega$ has finite Lebesgue measure.

\begin{theorem}[Existence of general quasiperiodic sets at infinity]
	The families $\mathscr{D}_{aqp}^{\ty}(n)$ and $\mathscr{D}_{tqp}^{\ty}(n)$ are nonempty for all $n\in\{2,3,\ldots\}\cup\{\infty\}$. More specifically, for any $D\leq 0$, there exist $\O_{aqp}\in\mathscr{D}_{aqp}^{\ty}(n)$ and $\O_{aqp}\in\mathscr{D}_{tqp}^{\ty}(n)$ with $\dim_B^{\phi}=D$ (for any $\phi>1$).
	Furthermore, in case $D<-N$ these sets are (necessarily) of finite Lebesgue measure, and else, of infinite Lebesgue measure. Moreover, in both cases they are $\phi$-shell Minkowski nondegenerate.
\end{theorem}

We omit the proof since it closely follows the construction from \cite[Section 5]{ra3} by adapting to new definitions and removing the assumption of $\Omega$ being of finite Lebesgue measure. We emphasize also that the general statement is in fact proven when in Definition \ref{quasiperiodic} we replace $B_{T,t}(0)$ by $K_{T}(0)\setminus K_t(0)$ with $K_t(0)$ being the ball of radius $t$ centered at zero in the $|\cdot|_\infty$-metric on $\eR^N$.

Similarly as in \cite[Section 5]{ra3} one can show that if we restrict $D\in (-N+1,0]$, the asymptotics of $|B_{T,t}(0)\cap\Omega|$ and $|(K_T(0)\setminus K_t(0))\cap\Omega|$ coincide up to the first order. Hence in this case we obtain also general quasiperiodic sets exactly in the sense of Definition \ref{quasiperiodic}.

\section{Concluding remarks and perspectives}

In this final section we give remarks on the connection between the notions of ($\phi$-shell) Minkowski content at infinity with two alternative but natural approaches when dealing with sets at infinity.
Namely, the first comes from the one-point compactification and the second from the notion of the surface Minkowski content.

\subsection{One-point compactification}\label{subsec:cpt}

When dealing with unbounded sets in $\eR^N$ it is natural to turn the one-point compactification and ask if there is a connection between the fractal properties of unbounded sets in $\eR^N$ at infinity and the classical fractal properties of their compactified images.

We choose the Riemann sphere $\mathbb{S}^N$ to be the unit sphere $$
\mathbb{S}^N:=\left\{(y_1,\ldots,y_{N+1})\,:\,\sum_{i=1}^{N+1}y_i^2=1\right\}
$$ with $\eR^N$ considered as the equatorial hyper-plane $\{y_{N+1}=0\}$.
Then, the one-point compactification is realized via the stereographic projection $\Psi\colon\eR^N\to\mathbb{S}^N$ defined by
\begin{equation}\label{stereo}
\Psi(x_1,\ldots,x_N):=\left(\frac{2x_1}{|x|^2+1},\ldots,\frac{2x_N}{|x|^2+1},\frac{|x|^2-1}{|x|^2+1}\right),
\end{equation}
where $x=(x_1,\ldots,x_N)\in\eR^N$,
while we let $\Psi(\ty):=\mathrm{\mathbf{N}}=(0,\ldots,0,1)$, the north pole.
It is easy to see that $\Psi(\eR^N)=\mathbb{S}^N\setminus\{\mathbf{N}\}$ is an immersed submanifold of $\eR^{N+1}$.
Hence, if $\O\subseteq\eR^N$, then the (spherical) $N$-dimensional volume of its image $\Psi(\O)\subseteq\mathbb{S}^N$ is given by
\begin{equation}\label{SNvolumen}
|\Psi(\O)|_{\mathbb{S}}=\int_{\O}\sqrt{\det\left((D\Psi)^{\tau}D\Psi\right)}\di x=\int_{\O}\frac{2^N}{(1+|x|^2)^N}\di x
\end{equation}
and we will call $|\Psi(\O)|_{\mathbb{S}}$ the {\em spherical $N$-dimensional volume} of $\O$.

Since for every Lebesgue measurable $\O\subseteq\eR^N$ its spherical volume is finite, it makes sense to define the notions of {\em spherical $($upper, lower$)$ Minkowski content} and {\em spherical $($upper, lower$)$ Minkowski dimension} analogously as in the classical case.
Here we will use the {\em spherical} $\delta$-neighborhood for subsets of $\mathbb{S}^N$:
\begin{equation}\label{sphericalneigh}
A_{\delta,\eS}:=\{y\in\mathbb{S}^N\,:\, d_{\eS}(y,A)<\delta\},
\end{equation}
where $A\subseteq\mathbb{S}^N$, $\d>0$ and $d_{\eS}$ is the induced spherical metric on $\mathbb{S}^N$.

\begin{definition}\label{spherical_cont}
The {\em upper spherical Minkowski contents} for $A\subseteq\mathbb{S}^N$ is defined in the usual way:
\begin{equation}\label{uppersphericalM}
{\ovb{\mathcal{M}}}_{\eS}^r(A):=\limsup_{\delta\to 0^+}\frac{|A_{\delta,{\eS}}|_{\mathbb{S}}}{\delta^{N-r}},
\end{equation}
and, analogously, the {\em lower} counterpart denoting it by $\unb{\mathcal{M}}_{\eS}^r(A)$.

Furthermore, the {\em upper (lower) spherical Minkowski dimension} of a set $A\subseteq\mathbb{S}^N$ is then introduced in a standard way
\end{definition}

The just introduced notions generalize to the case of relative fractal drums $(A,\O)$ with $A$ and $\O$ being subsets of $\eS^N$ naturally by replacing $|A_{\delta,{\eS}}|_{\mathbb{S}}$ by $|A_{\delta,{\eS}}\cap\O|_{\mathbb{S}}$ in Definition~\ref{spherical_cont} above.
We are now able to state and prove a comparison result for the Minkowski content of $\O$ at infinity and its spherical Minkowski content defined just above.

\begin{theorem}\label{sferni_sadrzaji}
Let $\O\subseteq\eR^N$ be a Lebesgue measurable set of finite measure and $\mathbf{N}$ the north pole of $\eS^N$. Then, for every $r<-N$ we have
\begin{equation}\label{Gsfsd}
-\frac{N+r}{N-r}\left(\frac{2N}{N-r}\right)^{-\frac{N+r}{N-r}}2^r{\ovb{\M}}^{r}(\ty,\O)\leq{\ovb{\mathcal{M}}}_{\eS}^{\,r}(\mathbf{N},\Psi(\O))\leq 2^r{\ovb{\M}}^{r}(\ty,\O)
\end{equation}
and
\begin{equation}\label{Dsfsd}
\unb{\mathcal{M}}_{\eS}^{r}(\mathbf{N},\Psi(\O))\leq 2^r{\unb{\M}}^{r}(\ty,\O).
\end{equation}
\end{theorem}

%

\begin{proof}
First, we note that elementary trigonometry yields that for any $\delta\in(0,\pi)$ we have
\begin{equation}\label{sferna_t_okolina}	
\Psi^{-1}(\{\mathbf{N}\}_{\delta,\eS}\cap A)=B_{\cot\frac{\delta}{2}}(0)^c\cap\Psi^{-1}(A).
\end{equation}
%
%

By \eqref{SNvolumen} and \eqref{sferna_t_okolina} 
 we have for $\delta\in(0,\pi)$ that
$$
\begin{aligned}
|\{\mathbf{N}\}_{\delta,\eS}\cap\Psi(\O)|_{\eS}=&\int_{B_{\cot{\frac{\d}{2}}}(0)^c\cap\O}\frac{2^N}{(1+|x|^2)^N}\di x
\leq\frac{2^N|B_{\cot\frac{\delta}{2}}(0)^c\cap\O|}{\left(1+\cot^2\frac{\d}{2}\right)^N}.
\end{aligned}
$$
Next, we introduce a new variable $t:=\cot(\d/2)$ and observe that $\delta\to0^+$ if and only if $t\to +\ty$.
Furthermore, for $r\in\eR$ from the above inequality we have 
$$
\begin{aligned}
\frac{|\{\mathbf{N}\}_{\delta,\eS}\cap\Psi(\O)|_{\eS}}{\d^{N-r}}\leq&\frac{2^N}{\left(1+\cot^2\frac{\d}{2}\right)^N}\,\frac{|B_{\cot\frac{\delta}{2}}(0)^c\cap\O|}{\d^{N-r}}=\frac{2^N}{\left(1+t^2\right)^N}\,\frac{|B_{t}(0)^c\cap\O|}{2^{N-r}(\arccot t)^{N-r}}\\
=&2^r\frac{t^{2N}}{\left(1+t^2\right)^N}\,\frac{1}{(t\arccot t)^{N-r}}\,\frac{|B_{t}(0)^c\cap\O|}{t^{N+r}}.
\end{aligned}
$$
Now, since $t\arccot t\to 1$ when $t\to +\ty$ we obtain the second inequality in \eqref{Gsfsd} and~\eqref{Dsfsd} by taking the upper and lower limit as $\delta\to 0^+$, respectively.

To prove the first inequality in~\eqref{Gsfsd} we fix $\phi>1$ and similarly as before we have
$$
\begin{aligned}
|\{\mathbf{N}\}_{\delta,\eS}\cap&\Psi(\O)|_{\eS}=\int_{B_{\cot{\frac{\d}{2}}}(0)^c\cap\O}\frac{2^N}{(1+|x|^2)^N}\di x\\
\geq&\int_{B_{\cot{\frac{\d}{2}}}(0)^c\cap B_{\phi\cot{\frac{\d}{2}}}(0) \cap\O}\frac{2^N}{(1+|x|^2)^N}\di x
\geq&\frac{2^N|B_{t,\phi t}(0)\cap\O|}{(1+\phi^2t^2)^N},
\end{aligned}
$$
where we have again introduced the variable $t:=\cot(\d/2)$.
This implies that for $r< -N$ we have
$$
\begin{aligned}
\frac{|\{\mathbf{N}\}_{\delta,\eS}\cap\Psi(\O)|_{\eS}}{\d^{N-r}}\geq&\frac{2^r}{(1+\phi^2t^2)^N}\,\frac{|B_{t,\phi t}(0)\cap\O|}{(\arccot t)^{N-r}}\\
=&2^r\frac{t^{2N}}{(1+\phi^2t^2)^N}\,\frac{1}{(t\arccot t)^{N-r}}\,\frac{|B_{t,\phi t}(0)\cap\O|}{t^{N+r}},
\end{aligned}
$$
hence, by taking the upper limit as $\delta\to 0^+$ we obtain
\begin{equation}
{\ovb{\mathcal{M}}}_{\eS}^{r}(\mathbf{N},\Psi(\O))\geq\frac{2^r}{\phi^{2N}}\,{\ovb{\M}}_{\phi}^{r}(\ty,\O)\geq\frac{2^r(1-\phi^{N+r})}{\phi^{2N}}{\ovb{\M}}^{r}(\ty,\O),
\end{equation}
where the second inequality follows from
Proposition~\ref{prstenasta}.

The optimal constant in \eqref{Gsfsd} is now easily obtained by maximizing the real function $\phi\mapsto\phi^{-2N}(1-\phi^{N+r})$ on the interval $(1,\ty)$.
\end{proof}

%

\begin{remark}
An open question is whether the inequalities in Theorem~\ref{sferni_sadrzaji} are sharp.
One could also state a variant of Theorem~\ref{sferni_sadrzaji} in terms of the $\phi$-shell Minkowski content at infinity and valid for unbounded sets having possibly infinite Lebesgue measure but for the sake of brevity we omit it here.
\end{remark}

\subsection{Surface Minkowski Content at Infinity}\label{surface_sec}

Now we will take a closer look into the possible connection between the $\phi$-shell Minkowski content at infinity and a natural notion of {\em surface Minkowski content at infinity} introduced just below. 
Inspired by Remark~\ref{the_remark} we now introduce the following definition.

\begin{definition}\label{surf_ty}
Let $\O$ be a Lebesgue measurable subset of $\eR^N$, and denote with $\mathscr{H}^{N-1}$ the $(N-1)$-dimensional normalized Hausdorff measure (i.e., scaled by the volume of the unit $(N-1)$-ball so that it coincides with the $(N-1)$-dimensional Lebesgue measure).
Then, for $r\in\eR$, we define the {\em $r$-dimensional upper surface Minkowski content of $\O$ at infinity} as
\begin{equation}
{\ovb{\mathcal{S}}}^{r}(\ty,\O):=\limsup_{t\to +\ty}\frac{\mathscr{H}^{N-1}(S_t(0)\cap\O)}{t^{N-1+r}}
\end{equation}
where $S_t(0)$ denotes the $(N-1)$-dimensional sphere of radius $t$ with center at $0$
and, analogously, we define the lower counterpart.
\end{definition}

The next result gives a generalization of a well known fact concerning differentiability of the volume of parallel sets to the current setting; see \cite{winter}.

\begin{proposition}
\label{prop:5.5}
Let $\O$ be a Lebesgue measurable subset of $\eR^N$.
Then, for Lebesgue a.e. $t>0$ we have that\footnote{Using formula \ref{deriva} along with the coarea formula for the mapping $x\mapsto|x|$ and integration by parts, one could now give an alternative proof of Proposition \ref{integralna_veza_inf}, we omit the details.}
\begin{equation}\label{deriva}
\frac{\di}{\di t}|B_t(0)\cap\O|=\mathscr{H}^{N-1}(S_t(0)\cap\O).
\end{equation}
Furthermore, if $|\O|<\ty$ then we have that
\begin{equation}\label{derivb}
\frac{\di}{\di t}|B_t(0)^c\cap\O|=-\mathscr{H}^{N-1}(S_t(0)\cap\O).
\end{equation}
\end{proposition}

\begin{proof}
We will use~\cite[Proposition~2.10]{mink}.
In short, this result states that for a closed subset $A$ of $\eR^N$ with Lebesgue measure equal to zero and a Lebesgue measurable\footnote{The original assumption in ~\cite[Proposition~2.10]{mink} was that $\O$ is open, but by going through the proof, it is clear that it is enough to assume that $\O$ is Lebesgue measurable.} subset $\O$ of $\eR^N$ we have that
\begin{equation}\label{A_t_derivacija}
\frac{\di}{\di t}|A_t\cap\O|=\mathscr{H}^{N-1}(\partial A_t\cap\O)
\end{equation}
for a.e. $t>0$.
This proves~\eqref{deriva} if we let $A:=\{0\}$.
Furthermore, since for $\O$ of finite Lebesgue measure we have that $|B_t(0)^c\cap\O|=|\O|-|B_t(0)\cap\O|$,~\eqref{deriva} implies~\eqref{derivb} in this case. 
\end{proof}

An immediate consequence is the following relation between the surface measure and the volume of the $\phi$-shell when $\phi$ shrinks to 1.

\begin{proposition}
Let $\O$ be a Lebesgue measurable subset of $\eR^N$.
Then, we have
\begin{equation}
\lim_{\phi\to 1^+}\frac{|B_{t,\phi t}(0)\cap\O|}{\log\phi}=t\,\mathscr{H}^{N-1}(S_t(0)\cap\O)
\end{equation}
for Lebesgue a.e. $t>0$.
\end{proposition}

\begin{proof}
Since $|B_{t,\phi t}(0)\cap\O|=|B_{\phi t}(0)\cap\O|-|B_{t}(0)\cap\O|$ and by letting $h=\log\phi$ we have
\begin{equation}\label{lim_lim}
\lim_{\phi\to 1^+}\frac{|B_{t,\phi t}(0)\cap\O|}{\log\phi}=\lim_{h\to 0^+}\frac{|B_{\E^{h}t}(0)\cap\O|-|B_{t}(0)\cap\O|}{h}=f'(\log t),
\end{equation}
where we have let
$
f(\tau):=|B_{\E^{\tau}}(0)\cap\O|.
$
Finally, by Proposition \ref{prop:5.5} we have that
$
f'(\tau)=\E^{\tau}\,\mathscr{H}^{N-1}(S_{\E^{\tau}}(0)\cap\O)
$ for a.e. $\tau\in\eR$ which completes the proof.
\end{proof}

In light of this, if we could justify the interchange of the order of taking the limit as $\phi\to 1^+$ and the upper limit  as $t\to +\ty$ for $(\ty,\O)$, we would get that
\begin{equation}\label{eq:surf-non}
{\ovb{\mathcal{S}}}^{r}(\ty,\O)=\limsup_{t\to +\ty}\frac{\mathscr{H}^{N-1}(S_t(0)\cap\O)}{t^{N-1+r}}=\lim_{\phi\to 1^+}\frac{{\ovb{\M}}^{r}_{\phi}(\ty,\O)}{\log\phi},
\end{equation}
and an analogous equality for the lower surface Minkowski content of $(\ty,\O)$.
Of course, the interchange above is not justified and the conditions when it can be made need to be investigated in future work.
Also, note that, a priori, the limit $\lim_{\phi\to 1^+}{{\ovb{\M}}^{r}_{\phi}(\ty,\O)}/{\log\phi}$ does not have to even exist.

To demonstrate a trivial situation when \eqref{eq:surf-non} holds, we revisit Example~\ref{er_N} where we have obtained the $\phi$-shell Minkowski content of $(\ty,\eR^N)$ and provide another simple example dealing with an unbounded strip of finite width. 

\begin{example}\label{diskusija}
Recall from Example \ref{er_N} that $\M_{\phi}^0(\ty,\eR^N)={\pi^{\frac{N}{2}}(\phi^N-1)}/{\Gamma\left(\frac{N}{2}+1\right)}$ and note that in this example we have
$
\lim_{\phi\to 1^+}\frac{\M_{\phi}^0(\ty,\eR^N)}{\log\phi}=\frac{N\pi^{\frac{N}{2}}}{\Gamma\left(\frac{N}{2}+1\right)}.
$
On the other hand, for the $0$-dimensional surface Minkowski content of $(\ty,\eR^N)$ we have
$
\mathcal{S}^0(\ty,\eR^N)=\lim_{t\to +\ty}\frac{\mathscr{H}^{N-1}(S_t(0))}{t^{N-1}}=\frac{N\pi^{\frac{N}{2}}}{\Gamma\left(\frac{N}{2}+1\right)}
$
so that \eqref{eq:surf-non} holds in this case.
\end{example}

\begin{example}\label{ex:strip}
Let $\O$ be a horizontal strip of finite height, i.e., let $h>0$ and
$$
\O:=\left\{(x,y)\in\eR^2\,:\,0<y<h\right\}.
$$
Then, for any $\phi>1$ and $t>h$ it is clear that we have
$$
2h(\sqrt{\phi^2t^2-h^2}-t)\leq|B_{t,\phi t}(0)\cap\O|\leq 2h(\phi t-\sqrt{t^2-h^2}),
$$
which implies that ${{\dim}}_B^{\phi}(\ty,\O)=-1$ and $\M_{\phi}^{-1}(\ty,\O)=2h(\phi-1)$.

%

Next, observe that $\mathscr{H}^{1}(S_t(0)\cap\O)=\sqrt{1+t^2}\arcsin(h/t)$, hence 
$$
\mathcal S^{-1}(\ty,\O)=\lim_{t\to +\ty}\frac{H^{1}(S_t(0)\cap\O)}{t^{2-1-1}}=2h=\lim_{\phi\to 1^+}\frac{\M_{\phi}^{-1}(\ty,\O)}{\log\phi}
$$
so that \eqref{eq:surf-non} holds.
\end{example}

As a final comment, we point out that one would like to establish analogous relations between the relative Minkowski content and the corresponding relative surface Minkowski content, both for classical RFDs $(A,\O)$ and for $(\infty,\Omega)$ as was done in~\cite{winter} for the non-relative case.
An obstacle that one needs to overcome to achieve this is the fact that for a relative fractal drum $(A,\O)$, its relative tube function $t\mapsto |A_t\cap\O|$ is not a Kneser function\footnote{Recall that a function $f\colon(0,\ty)\to(0,\ty)$ is called a {\em Kneser function of order $r\geq 1$}, if for all $0<a\leq b<\ty$ and $\phi\geq 1$ we have
$
f(\phi b)-f(\phi a)\leq \phi^r(f(b)-f(a)).
$} of order $N$, in general; see \cite[Example 4.111]{ra}.
The crucial fact that enabled to prove that a bounded set $A\subseteq\eR^N$ is Minkowski nondegenerate if and only if it is surface Minkowski nondegenerate in~\cite{winter} was that for a such a set its tube function $t\mapsto |A_t|$ is a Kneser function of order $N$ (see~\cite{Kne}).

\end{document}